\documentclass{amsart} 
 
\usepackage{amssymb,amsmath,bbm} 
\usepackage{stmaryrd,hyperref} 
\usepackage{graphics,graphicx,subfigure} 
 
\usepackage[latin1]{inputenc} 
 
\newtheorem{theo}{Theorem} 
\newtheorem{lem}[theo]{Lemma} 
\newtheorem{prop}[theo]{Proposition} 
\newtheorem{claim}[theo]{Claim} 
\newtheorem{Def}[theo]{Definition} 
\newtheorem{remark}[theo]{Remark} 
 
\newcommand{\ite}{\noindent $\bullet$}

\newcommand{\beq}{\vspace{-.1cm}$$} 
\newcommand{\eeq}{\vspace{-.1cm}$$}

\newcommand\ds{\displaystyle} 
\newcommand\al{\alpha} 
\newcommand\be{\beta} 
\newcommand\la{\lambda}

\newcommand{\mO}{\mathcal{O}} 
\newcommand{\wO}{\widetilde{\mO}} 
\newcommand{\mB}{\mathcal{B}} 
\newcommand{\wB}{\widetilde{\mB}} 
\newcommand{\wS}{\widetilde{\mS}}

\newcommand{\mE}{\mathcal{E}} 
\newcommand{\cB}{\mathcal{B}} 
 
\newcommand{\mC}{\mathcal{C}} 
 
\newcommand{\mS}{\mathcal{S}} 
\newcommand{\mG}{\mathcal{G}} 
\newcommand{\NN}{\mathbb{N}} 
 
\newcommand{\wOO}{\wO_0} 
\newcommand{\wSO}{\wS_0} 

\newcommand{\vv}{\textrm{v}} 
\newcommand{\ee}{\textrm{e}} 
\newcommand{\ff}{\textrm{f}} 
\newcommand{\bb}{\textrm{b}} 
\newcommand{\ww}{\textrm{w}} 
 
\newcommand{\oo}{\textrm{o}}

\newcommand{\bu}{\bullet} 
\newcommand{\ddm}{$d/(d\!-\!2)$} 
\newcommand{\bbm}{$b/(b\!-\!1)$} 
 
\newcommand{\fig}[3]{\begin{figure}[h]\begin{center}\includegraphics[#1]{#2}\end{center}\caption{#3}\label{fig:#2}\end{figure}} 

\def\hA{\hat{A}} 
\def\hS{\hat{S}} 
\def\cA{\mathcal{A}} 
\def\cM{\mathcal{M}} 
 
\def\cW{\mathcal{W}} 
 
\newcommand{\titre}[1]{\noindent \textbf{#1}} 
 
\author[O. Bernardi and \'E. Fusy]{Olivier Bernardi$^{*}$ \and \'{E}ric Fusy$^{\dagger}$} 
\thanks{$^{*}$Dept. of Mathematics, MIT, 77 Massachusetts Avenue, Cambridge MA 02139, USA, 
bernardi@math.mit.edu. Supported by the French ANR project A3 and the European project ExploreMaps -- ERC StG 208471.\\ 
$^{\dagger}$LIX, \'Ecole Polytechnique, 91128 Palaiseau cedex, France, fusy@lix.polytechnique.fr. 
Supported by the European project 
ExploreMaps -- ERC StG 208471}

\title[A bijection for triangulations, quadrangulations, etc.]{A bijection for triangulations, quadrangulations, pentagulations, etc.}

\begin{document} 
\maketitle 
\begin{abstract} 
A $d$-angulation is a planar map with faces of degree $d$. We present for each integer $d\geq 3$ a bijection between the class of $d$-angulations of girth~$d$ (i.e., with no cycle of length less than $d$) and a class of decorated plane trees. Each of the bijections is obtained by specializing a ``master bijection'' which extends an earlier construction of the first author. 
Our construction unifies known bijections by Fusy, Poulalhon and Schaeffer for triangulations ($d=3$) and by Schaeffer for quadrangulations ($d=4$). For $d\geq 5$, both the bijections and the enumerative results are new. 
 
We also extend our bijections so as to enumerate \emph{$p$-gonal $d$-angulations} ($d$-angulations with a simple boundary of length $p$) of girth $d$. We thereby recover bijectively the results of Brown for simple $p$-gonal triangulations and simple $2p$-gonal quadrangulations and establish new results for $d\geq 5$. 
 
A key ingredient in our proofs is a class of orientations characterizing $d$-angulations of girth $d$. Earlier results by Schnyder and by De Fraysseix and Ossona de Mendez showed that simple triangulations  and simple quadrangulations  are characterized by the existence of orientations having respectively indegree 3 and 2 at each inner vertex. 
We extend this characterization by showing that a $d$-angulation has girth $d$ if and only if the graph obtained by duplicating each edge $d-2$ times admits an orientation having indegree $d$ at each inner vertex. 
\end{abstract}

\section{Introduction} 
The enumeration of planar maps (connected planar graphs embedded in the sphere) has received a lot of 
attention since the seminal work of Tutte in the 60's~\cite{Tu63}. Tutte's method for counting a class of maps consists in translating a recursive description of the class (typically obtained by deleting an edge) into a functional equation 
satisfied by the corresponding generating function. The translation usually requires to introduce a ``catalytic'' variable, and the functional equation is solved using the so-called ``quadratic method''~\cite[Sec.2.9]{GoJa83} 
or its extensions~\cite{BJ06a}. The final result is, for many classes of maps, a strikingly simple counting formula. 
For instance, the number of maps with $n$ edges is $\frac{2\cdot 3^n}{(n+1)(n+2)}{2n\choose n}$. 
Tutte's method has the advantage of being systematic, but is quite technical 
in the way of solving the equations and does not give a combinatorial understanding of the simple-looking enumerative formulas. Other methods for the enumeration of maps were developed later, based either on matrix integrals, representations of the symmetric group, or bijections. 
 
The bijective approach has the advantage of giving deep insights into the combinatorial properties of maps. This approach was used to solve several statistical physics models (Ising, hard particles) on random lattices~\cite{BoSchaeffe,BoDiGu07}, and to investigate the metric properties of random maps~\cite{ChSc04,BDFG:mobiles}. It also has nice algorithmic applications  (random generation and asymptotically optimal encoding in linear time). 
The first bijections appeared in~\cite{CoriVa}, and later in~\cite{Schaeffer:these} where more transparent constructions were given for several classes of maps. 
Typically, bijections are from a class of ``decorated'' plane trees to a class of maps and operate on trees by progressively 
closing facial cycles. However, even if the bijective approach was successfully applied to many classes of maps, it is not, up to now, as systematic as Tutte's recursive method. Indeed, for each class of maps, one has to guess (using counting results) the corresponding family of decorated trees, and then to invent a mapping between trees and maps. 
 
A contribution of this article is to make the bijective approach more systematic by showing that a single ``master bijection'' can be specialized in various ways so as to obtain bijections for several classes of maps. 
The master bijection $\Phi$ presented in Section~\ref{section:mobile} is an extension of a construction by the first author~\cite{OB:boisees} (reformulated 
and extended to higher genus in~\cite{OB:covered-maps}).  To be more exact, the bijection $\Phi$ is between a set $\mO$ of \emph{oriented maps} and a set of mobiles. 
It turns out that for many classes of maps there is a canonical way of orienting the maps in the class $\mC$, so that $\mC$ is identified with a subfamily $\mO_\mC$ of $\mO$ on which our bijection $\Phi$ restricts nicely. 
Typically, the oriented maps in $\mO_\mC$ are characterized by degree constraints which can be traced through our construction and yield a degree characterization of the associated mobiles. Then the family of mobiles can be  specified by a decomposition grammar and is amenable to the Lagrange inversion formula for counting. 
To summarize, the bijective approach presented here is more systematic, since it consists in specializing the master bijection $\Phi$ to several classes of maps. 
The problem of enumerating a class of map $\mC$ then reduces to guessing a family of ``canonical'' orientations (in $\mO$) for $\mC$, instead of guessing a family of decorated trees and a bijection, which is harder.

We apply our bijective strategy to an infinite collection of classes of maps. More precisely, we consider for each integer $d\geq 3$ the class $\mC_d$ of $d$\emph{-angulations} (planar maps with faces of degree $d$) of girth $d$ (i.e., having no cycle of length less than $d$). The families $\mC_3$ and $\mC_4$ correspond respectively 
to simple triangulations and simple quadrangulations. We show in Section~\ref{sec:bij_dang} that each class $\mC_d$ is amenable to our bijective strategy: by specializing the bijection $\Phi$, we obtain a bijection between the class $\mC_d$ and a class of mobiles characterized by certain degree conditions.  Bijections already existed for the class $\mC_3$ of simple triangulations~\cite{Poulalhon:triang-3connexe+boundary,FuPoScL} and the class $\mC_4$ of simple quadrangulations~\cite{Fusy:these,Schaeffer:these}.  Our approach actually coincides (and unifies) the bijections presented respectively in~\cite[Theo.4.10]{FuPoScL} for triangulations and in~\cite[Sec.2.3.3]{Schaeffer:these} for quadrangulations. 
 
In Section~\ref{sec:bij_dang_bond}, we show that our bijective strategy applies more generally to $d$-angulations with a simple boundary. More precisely, for $p\geq d\geq 3$, we consider the class $\mC_{p,d}$ of $p$-gonal $d$-angulations (maps with inner faces of degree $d$ and an outer face which is a simple cycle of degree $p$) of girth $d$. We specialize the master bijection in order to get a bijection for the class $\mC_{p,d}$ (this requires to first mark an inner face and decompose the map into two pieces; moreover one needs to use two versions of the master bijection denoted $\Phi_+$ and $\Phi_-$). Again the associated mobiles are characterized by simple degree conditions.

For $p\geq d \geq 3$, the mobiles associated to the classes $\mC_{p,d}$ (and in particular to the classes $\mC_{d}\equiv \mC_{d,d}$) are characterized by certain degree conditions, hence their generating functions are specified by a system of algebraic equations. 
We thus obtain bijectively an algebraic system characterizing the generating function of the class $\mC_{p,d}$ (these results are new for $d\geq 5$). These algebraic systems follow a simple uniform pattern, and their forms imply that the number of rooted $d$-angulations of girth $d$ of size $n$ is asymptotically equal to $c_d\, n^{-5/2}{\gamma_d}^n$ where $c_d$ and $\gamma_d$ are computable constant (a typical behavior for families of maps).  
In the cases $d\in \{3,4\}$, we recover  bijectively the enumerative results obtained by Brown~\cite{Brown:triang3connexes+boundary,Brown:quadrangulation+boundary}. That is, we show that  the number $t_{p,n}$ of $p$-gonal simple triangulations  with $p+n$ vertices and the number  $q_{p,n}$ of $2p$-gonal simple quadrangulations with $2p+n$ vertices are 
\beq 
t_{p,n}=\!\frac{2(2p-3)!}{(p-1)!(p-3)!} \frac{(4n+2p-5)!}{n!(3n+2p-3)!},\ \ \ \   q_{p,n}=\!\frac{3(3p-2)!}{(p-2)!(2p-1)!}\frac{(3n+3p-4)!}{n!(2n+3p-2)!}. 
\eeq 
As mentioned above, bijective proofs already existed for the case $p=d\in\{3,4\}$ \cite{Poulalhon:triang-3connexe+boundary,FuPoScL,Fusy:these,Schaeffer:these}. Moreover a bijection different from ours  was given for the cases $p\geq d=3$ in~\cite{Poulalhon:triang-3connexe+boundary},  but the surjectivity was proved with the help of Brown's counting formulas. 

We now point out one of the key ingredients used in our bijective strategy. As explained above, for a class $\mC$ of maps, our bijective strategy requires to specify a canonical orientation in order to identify $\mC$ with a subclass $\mO_\mC\subset \mO$ of oriented maps. For $d$-angulations with a simple outer face, a simple calculation based on the Euler relation shows that the numbers $n$ and $m$ of inner vertices and inner edges are related by $m=\tfrac{d}{d-2} n$. This suggests a candidate canonical orientation for a $d$-angulation $G$: just ask for an orientation with indegree $\tfrac{d}{d-2}$ at each inner vertex\ldots or more reasonably an orientation of the graph $(d-2)\cdot G$ (the graph obtained from $G$ replacing each edge by $(d-2)$ parallel edges) with indegree $d$ at each inner vertex. We show that such an orientation (conveniently formulated as a \emph{weighted 
orientation} on $G$) exists for a $d$-angulation if and only if it has girth~$d$, and use these orientations to identify the class $\mC_d$ with a subfamily $\mO_{\mC_d}\subset \mO$. This extends earlier results by Schnyder for simple triangulations~\cite{Schnyder:wood1} and by de Fraysseix and Ossona de Mendez for simple quadrangulations~\cite{Fraysseix:Topological-aspect-orientations}.\\ 
 
This article lays the foundations for some articles to come, where the master bijection strategy will be applied. 
In particular we will show in~\cite{BeFu_Girth} that specializing the master bijection $\Phi$ gives a bijective way of counting maps with control on the face-degrees and on the girth. This extends the known bijections~\cite{BDFG:mobiles,bouttier-2002-645} which give a bijective way of counting maps with control on the face-degrees but no constraint of girth (the bipartite case of the bijection in~\cite{BDFG:mobiles} was already recovered in~\cite{OB:covered-maps} as a specialization of the construction of~\cite{OB:boisees} which is related to $\Phi$).

\medskip

\section{Maps and orientations}\label{section:definitions} 
In this section we gather definitions about maps and orientations.\\ 
 
\titre{Maps.} 
A (planar) \emph{map} is a connected planar graph embedded in the oriented sphere and considered up to continuous deformation. Loops and multiple edges are allowed. 
The \emph{faces} are the connected components of the complementary of the graph. A \emph{plane tree} is a map without cycles (it has a unique face). Cutting an edge~$e$ at its middle point gives two \emph{half-edges}, each incident to an endpoint of $e$ (they are both incident to the same vertex if $e$ is a loop). We shall also consider some maps decorated with dangling half-edges called \emph{buds} (see e.g. Figure~\ref{fig:dual_opening}). A \emph{corner} is the angular section between two consecutive half-edges around a vertex. The degree of a vertex or face is the number of incident corners. A  $d$-\emph{angulation} is a map such that every face has degree~$d$. \emph{Triangulations} and \emph{quadrangulations} correspond to the cases $d=3$ and $d=4$ respectively. The \emph{girth} of a map is the minimal length of its cycles. Obviously a $d$-angulation has girth at most $d$ (except if it is a tree). A map is \emph{loopless} if it has girth at least $2$, and \emph{simple} (no loops nor multiple edges) if it has girth at least 3. 
 
The numbers $\vv$, $\ee$, $\ff$ of vertices, edges and faces of a map are related by the \emph{Euler relation}: $\vv-\ee+\ff=2$. We shall also use the following result. 
\begin{lem} \label{lem:counting} 
If a map with $\vv$ vertices and $\ee$ edges has a face $f_0$ of degree $p$ and the other faces of degree at least~$d$, then 
$$(d-2)(\ee-p) ~\leq~ d (\vv-p)+p-d,$$ 
with equality if and only if the faces other than $f_0$ have degree exactly $d$. 
\end{lem} 
\begin{proof} 
The incidence relation between faces and edges gives $2\ee\geq p+d(\ff-1)$, with equality if the faces other than $f_0$ have degree exactly $d$. Combining this with the  Euler relation (so as to eliminate $\ff$) gives the claim. 
\end{proof} 
 
A map is said to be \emph{vertex-rooted}, \emph{face-rooted}, or \emph{corner-rooted} respectively if one of its vertices, faces, or corners  is marked\footnote{Corner-rooted map are usually simply called \emph{rooted maps} in the literature. A face-rooted map can be thought as a \emph{plane map} (a connected graph embedded in the plane) by thinking of the root-face as the infinite face.}. The marked vertex, face or corner is called the \emph{root-vertex}, \emph{root-face}, or \emph{root-corner}. For a corner-rooted map, the marked corner is indicated by a dangling half-edge pointing to that corner; see Figure~\ref{fig:close_mobile_ccw}(c). 
A corner-rooted map is said to \emph{induce} the vertex-rooted map (resp. face-rooted map) obtained by marking the vertex incident to the root-corner (resp. the face containing the root-corner), but otherwise forgetting the root-corner. 
Given a face-rooted (or corner-rooted) map, the root-face is classically taken as the outer (infinite) face
in plane embeddings. The degree of the root-face is called the \emph{outer degree}.  
 The faces distinct from the root-face are called \emph{inner faces},
and the vertices and edges
 are said to be \emph{outer} or \emph{inner} whether they are incident to the root-face or not.
Similarly, a half-edge is said to be \emph{inner} (resp. \emph{outer}) if it lies on an inner (resp. outer) edge.\\

\titre{Orientations, biorientations.} 
Let $G$ be a map. An \emph{orientation} of $G$ is a choice of a direction for each edge of $G$. A \emph{biorientation} of $G$ is a choice of a direction for each half-edge of $G$: each half-edge can be either \emph{outgoing} (oriented toward the middle of the edge) or \emph{ingoing} (oriented toward the incident vertex). For $i=0,1,2$ we call $i$\emph{-way} an edge with exactly $i$ ingoing half-edges. Our convention for representing $i$-way edges is shown in Figure~\ref{fig:biorientation}(a). Clearly, orientations can be seen as a special kind of biorientations, in which every edge is 1-way. The \emph{indegree}  (resp. \emph{outdegree}) of a vertex $v$ is the number of ingoing (resp. outgoing) half-edges incident to $v$. The \emph{clockwise-degree} of a face $f$ is the number of outgoing half-edges $h$ incident to $f$ having the face $f$ on their right (i.e., the number of incidences  with 0-way edges, plus the number of 1-way edges having $f$ on their right). A face of clockwise-degree 3 is represented in Figure~\ref{fig:clockwise-minimal}(b). 
A \emph{directed path} in a biorientation is a  path going through some vertices $v_1,\ldots,v_{k+1}$ in such a way that for all $i=1,\ldots,k$ the edge between $v_{i}$ and $v_{i+1}$ is either 2-way or 1-way oriented toward $v_{i+1}$. A \emph{circuit} is a  \emph{directed cycle} (closed directed path) which is \emph{simple} (no edge or vertex is used twice). 
 
\fig{width=\linewidth}{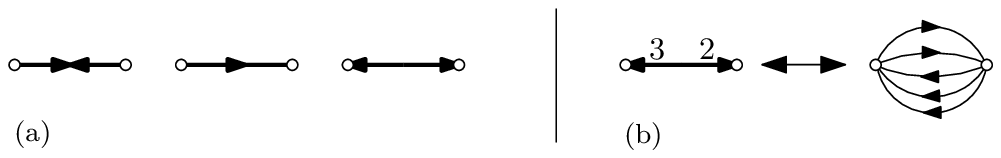}{(a) Biorientations: a 0-way edge (left), a 1-way edge (middle) and a 2-way edge (right). (b) Correspondence between $\NN$-biorientations of a map $M$ and ordinary orientations of the associated map $M'$ (obtained by replacing the edge $e$ by a bunch of parallel edges).}

A \emph{planar biorientation} (shortly called a biorientation hereafter) is a planar map  with a biorientation. 
A biorientation is said to be \emph{vertex-rooted}, \emph{face-rooted}, \emph{corner-rooted}  if the underlying map is 
respectively vertex-rooted, face-rooted, corner-rooted. A biorientation is \emph{accessible from a vertex} $v$ if any vertex is reachable from $v$ by a directed path. A  vertex-rooted (or corner-rooted) biorientation is said to be \emph{accessible} if it is accessible from the root-vertex. A circuit of a face-rooted (or corner-rooted) biorientation is \emph{clockwise} if the root-face is on its left and \emph{counterclockwise} otherwise; see Figure~\ref{fig:clockwise-minimal}(a). The biorientation is \emph{minimal} if it has no counterclockwise circuit.\\

\titre{Weighted biorientations.} 
A biorientation is \emph{weighted}  by associating a \emph{weight} $w(h)\in\mathbb{R}$ to each half-edge $h$. 
The \emph{weight} of an edge is the sum of the weights of the two half-edges. 
The \emph{weight} of a vertex~$v$ is the sum of the weights of the incident ingoing half-edges. The \emph{weight} of a face $f$ is the sum of the weights of the outgoing half-edges incident to $f$ and having $f$ on their right.  A face of weight -5 is represented in Figure~\ref{fig:clockwise-minimal}(b).
 
A weighted biorientation is called \emph{consistent} if the weights of ingoing half-edges are positive 
and the weights of outgoing half-edges are nonpositive (in that case the biorientation 
is determined by the weights). An \emph{$\NN$-biorientation} is a consistent 
weighted biorientation with weights in $\NN=\{0,1,2\ldots\}$ (so outgoing half-edges have weight $0$ 
and ingoing half-edges have positive integer weights). Note that ordinary orientations 
identify with $\NN$-biorientations where each edge has weight $1$. 
Let $M=(V,E)$ be a map, let $\alpha$ be a function from the vertex set $V$ to $\NN$ and let $\beta$ be a function from the edge set $E$ to $\NN$. 
 We call $\alpha/\beta$\emph{-orientation} an $\NN$-biorientation such that every vertex $v$ has weight $\alpha(v)$ and every edge $e$ has weight $\beta(e)$.  We now give a criterion 
 for the existence (and uniqueness) of a minimal $\alpha/\beta$-orientation. 
\begin{lem}\label{lem:unique_minimal} 
Let $M$ be a map with vertex set $V$ and edge set $E$, let $\alpha$ be a function from  $V$ to $\NN$, and let $\beta$ be a function from $E$ to $\NN$. 
If there exists an $\alpha/\beta$-orientation for $M$, then there is a unique minimal one. 
\end{lem} 
 
\begin{lem}\label{lem:exists-alpha} 
Let $M$ be a map with vertex set $V$ and edge set $E$, let $\alpha$ be a function from  $V$ to $\NN$, and let $\beta$ be a function from $E$ to $\NN$. The map $M$ admits an $\alpha/\beta$-orientation if and only if 
\begin{enumerate} 
\item[(i)] $\sum_{v\in V}\alpha(v)=\sum_{e\in E}\beta(e)$, 
\item[(ii)] for every subset $S$ of vertices,  $\sum_{v\in S}\alpha(v)\geq \sum_{e\in E_S}\beta(e)$   where $E_S$ is the set of edges with both ends in $S$. 
\end{enumerate} 
Moreover, $\alpha$-orientations are accessible from a vertex $u$ if and only if 
\begin{enumerate} 
\item[(iii)] for every nonempty subset $S$ 
of vertices not containing $u$, $\sum_{v\in S}\alpha(v)> \sum_{e\in E_S}\be(e)$. 
\end{enumerate} 
\end{lem}

Lemmas \ref{lem:unique_minimal} and \ref{lem:exists-alpha} are both known for the special case where the function $\beta$ takes value 1 on each edge of $G$ (this case corresponds to ordinary orientations and the weight of a vertex is equal to its indegree); see \cite{Felsner:lattice}. The proofs below are simple reductions to the case $\beta=1$.

\begin{proof}[Proof of Lemmas \ref{lem:unique_minimal} and  \ref{lem:exists-alpha}] 
We call $\beta$\emph{-weighted orientations} the $\NN$-biorientations of $M$ such that every edge $e$ 
has weight $\beta(e)$. 
Let $M'$ be the map obtained by replacing each edge $e$ of $M$ by $\be(e)$ edges in parallel. We call \emph{locally minimal} the (ordinary) orientations of $M'$ such that for any edge $e$ of $M$, the corresponding $\beta(e)$ parallel edges of $M'$ do not contain any counterclockwise circuit.  Clearly, the  $\beta$-weighted orientations of $M$ are in bijection with the \emph{locally minimal orientations} of $M'$, see Figure \ref{fig:biorientation}(b).  Thus, there exists an $\alpha/\beta$-orientation of $M$ if and only if there exists an ordinary orientation of $M'$ with indegree $\alpha(v)$ at each vertex. Therefore, Conditions (i),(ii) of Lemma \ref{lem:exists-alpha} for $M$ immediately follow from the same conditions for the particular case $\beta=1$ applied to the map $M'$ (the case $\beta=1$ is treated in \cite{Felsner:lattice}). Moreover an $\alpha/\beta$-orientation of $M$ is accessible from a vertex $u$ if and only if the corresponding orientation of $M'$ is accessible from $u$. Hence Condition (iii) also follows from the particular case $\beta=1$ applied to the map $M'$. Lastly, a $\beta$-weighted orientation of $M$ is minimal if and only if the corresponding orientation of $M'$ is minimal. Hence Lemma \ref{lem:exists-alpha} immediately follows from the  existence and uniqueness of a minimal orientation in the particular case $\beta=1$ applied to the map $M'$. 
\end{proof}

We now define several important classes of biorientations. A face-rooted biorientation is called \emph{clockwise-minimal} if it is minimal and each outer edge is either $2$-way or is $1$-way with a non-root face 
on its right (note that the contour of the root-face is not necessarily simple and can even contain isthmuses); see Figure \ref{fig:clockwise-minimal}(c). A clockwise-minimal biorientation is called \emph{accessible}  if it is accessible from one of the outer vertices (in this case, it is in fact accessible from any outer vertex because one can walk in clockwise order around the root-face). 
A biorientation (weighted or not) is called 
\emph{admissible} if it satisfies the following three conditions: 
(i)~the root-face contour is a simple cycle, (ii)~each outer vertex has indegree $1$, (iii)~each outer  edge is $1$-way, and the weights (if the biorientation is weighted) are $0$ and $1$ on respectively the outgoing 
and the ingoing half-edge; see Figure \ref{fig:clockwise-minimal}(d).
Note that, for such a biorientation the root-face is a circuit, and each inner half-edge incident to an outer vertex is outgoing. 
 
\begin{remark}\label{rk:minimal_Norientation}
For an $\NN$-biorientation, Conditions~(ii) and (iii) are equivalent to ``each outer vertex and each outer edge has weight $1$''.
\end{remark} 
 
 
\begin{Def}\label{def:mBwB} 
We denote by $\mO$ the set of clockwise-minimal accessible (ordinary) orientations, and by $\wO$ 
the subset of orientations in $\mO$ that are admissible. 
 
Similarly we denote by $\mB$ the set of clockwise-minimal accessible weighted biorientations, and by $\wB$ 
the subset of weighted biorientations in $\mB$ that are admissible. 
\end{Def}

 
 
\fig{width=0.8\linewidth}{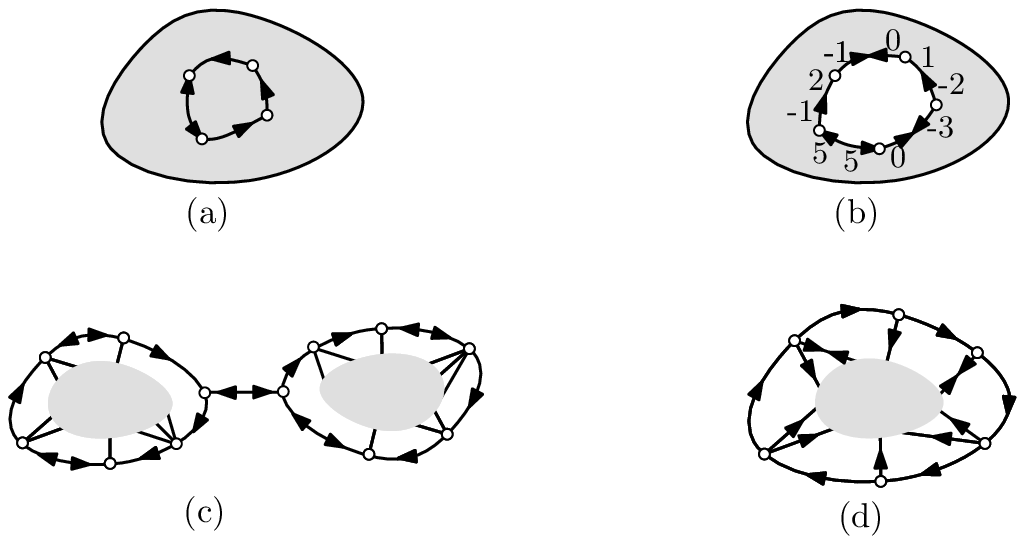}{(a) A counterclockwise circuit (here the root-face is the infinite face). (b) An inner face with clockwise degree 3 and weight $w=-3-1-1=-5$. (c) Situation around the root-face for a weighted biorientation in $\mB$. (d) Situation around the root-face for a weighted biorientation in $\wB$.}

\bigskip 
\section{Master bijections between oriented maps and mobiles}\label{section:mobile} 
In this section we describe the two ``master bijections'' $\Phi_-$, $\Phi_+$. These bijections will be specialized in Sections~\ref{sec:bij_dang} and~\ref{sec:bij_dang_bond} in order to obtain bijections for  $d$-angulations of girth $d$ with and without boundary.\\ 
 
\subsection{Master bijections $\Phi_+$,  $\Phi_-$ for ordinary orientations}~\\ 
A \emph{properly bicolored} plane tree is an unrooted plane tree with vertices colored black and white, with every edge joining a black and a white vertex. A \emph{properly bicolored mobile} is a properly bicolored plane tree where black vertices can be incident to some dangling half-edges called \emph{buds}. Buds are represented
by outgoing arrows in the figures. An example is shown in Figure~\ref{fig:dual_opening} (right). The \emph{excess} of a properly bicolored  mobile is the number of edges minus the number of buds. 
 
We now define the mappings $\Phi_\pm$ using a local operation performed around each edge. 
\begin{Def}\label{def:one-way-operation} 
Let $e$ be an edge of a planar biorientation, made of half-edges $h$ and $h'$ incident to vertices $v$ and $v'$ respectively. Let $c$ and $c'$ be the corners preceding $h$ and $h'$ in clockwise order around $v$ and $v'$ respectively, and let $f$ and $f'$ be the faces containing these corners. 
Let $b_f$, $b_f'$ be vertices placed inside the faces $f$ and $f'$ (with $b_f=b_{f'}$ if $f=f'$). If $e$ is a 1-way edge with $h$ being the ingoing half-edge  (as in Figure~\ref{fig:one-way-operation}), then the  \emph{local transformation} of $e$ consists in creating an edge joining the black vertex $b_f$ to the vertex $v$ in the corner $c$,  gluing a bud to $b_{f'}$ in the direction of $c'$, and deleting the edge $e$. 
\end{Def} 
The local transformation of a 1-way edge is illustrated in Figure~\ref{fig:one-way-operation} (ignore the weights $w,w'$ for the time being). 
 
\fig{width=9cm}{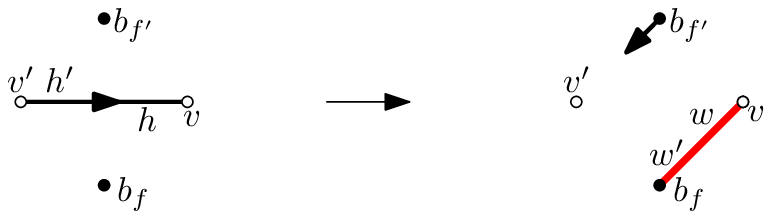}{The local transformation of a 1-way edge.}

\begin{Def}\label{def:master-bijections} 
Let $O$ be a face-rooted orientation in $\mO$ with root-face $f_0$. We view the vertices of $O$ as 
\emph{white} and place a \emph{black} vertex $b_f$ in each face $f$ of $O$. 
\begin{itemize} 
\item The embedded graph $\Phi_+(O)$ is obtained by performing the local transformation of each edge of $O$, and then deleting the black vertex $b_{f_0}$ and all the incident buds (the vertex $b_{f_0}$ is incident to no edge). 
\item If $O$ is in $\wO$, the embedded graph $\Phi_-(O)$ with black and white vertices is obtained by first returning all the edges of the outer face (which is a clockwise circuit), then performing the local transformation of each edge, and lastly deleting the black vertex $b_{f_0}$, the white outer vertices of $O$, and the edges between them (no other edge or bud is incident to these vertices). 
\end{itemize} 
\end{Def} 
 
\begin{figure} 
\begin{center} 
\includegraphics[width=\linewidth]{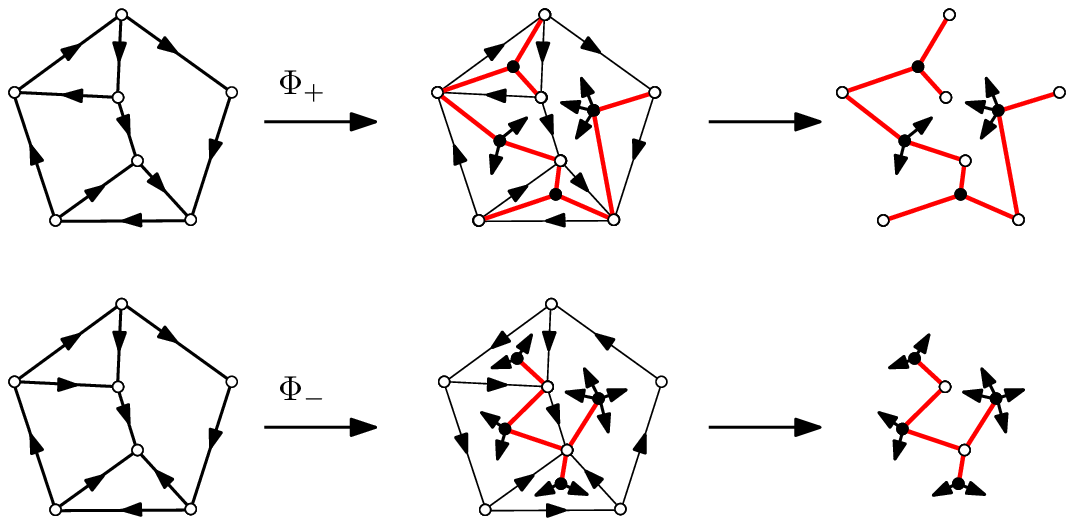} 
\end{center} 
\caption{Top: Mapping $\Phi_+$ applied to an orientation in $\mO$. Bottom: Mapping $\Phi_-$  applied to an orientation in  $\wO$. The vertex $b_{f_0}$ and the incident buds or edges are not represented.} 
\label{fig:dual_opening} 
\end{figure}

The mappings $\Phi_+$, $\Phi_-$ are illustrated in Figure~\ref{fig:dual_opening}. The following theorem is partly based on the results in~\cite{OB:boisees} and its proof is delayed to Section~\ref{sec:fromOBtoPhi}.

\begin{theo}\label{thm:master-bijections} 
Recall Definition~\ref{def:mBwB} for the sets $\mO$ and $\wO$ of orientations. 
\begin{itemize} 
\item 
 The mapping $\Phi_+$ is a bijection between $\mO$ and properly bicolored mobiles of positive excess; the outer degree is mapped to the excess of the mobile. 
\item 
The mapping $\Phi_-$ is a bijection between $\wO$ and properly bicolored mobiles of negative excess; the outer degree is mapped to minus the excess of the mobile. 
\end{itemize} 
\end{theo}

\subsection{Master bijections $\Phi_+$, $\Phi_-$ for  weighted biorientations} 
In this subsection we extend the bijections $\Phi_\pm$ to weighted biorientations. 
 
A \emph{mobile} is a plane tree with vertices colored either black or white (the coloring is not necessarily proper), and where black vertices can be incident to some dangling half-edges called \emph{buds}  (buds
are represented by outgoing arrows in the figures). The \emph{excess} of a  mobile is the number of black-white edges, plus twice the number of white-white edges,  minus the number of buds. A  mobile is \emph{weighted} if a weight in $\mathbb{R}$ is associated to each non-bud half-edge. The \emph{indegree} of a vertex $v$ (black or white) is the number of incident non-bud half-edges, and its \emph{weight} is the sum of weights of the incident non-bud half-edges. The \emph{outdegree} of a black-vertex $b$ is the number of incident buds.


We now define the mappings $\Phi_{\pm}$ for weighted biorientations using local operations performed around each edge. We consider an edge $e$ of a weighted biorientation and adopt the notations $h$, $h'$, $v$, $v'$, $c$, $c'$, $f$, $f'$, $b_f$, $b_{f'}$ of Definition~\ref{def:one-way-operation}. We also denote by $w$ and $w'$ respectively the weights of the half-edges $h$ and $h'$. 
Definition~\ref{def:one-way-operation} (\emph{local transformation} of a 1-way edge) is supplemented with weights attributed to the two halves of the edge created between $v$ and $b_f$:  the half-edge incident to $v$ receives weight $w$ and the half-edge incident to $b_f$ receives weight $w'$. 
If the edge $e$ is 0-way, the \emph{local transformation} of $e$ consists in creating an edge between $b_f$ and $b_{f'}$ with weight $w'$ for the half-edge incident to $b_f$ and weight $w$ for the half-edge incident to $b_{f'}$, and then deleting $e$; see Figure \ref{fig:i-way-operation}(a). If the edge $e$ is 2-way, the \emph{local transformation} of $e$ consists in creating buds incident to $b_f$ and $b_{f'}$ respectively in the direction of the corners $c$ and $c'$, and leaving intact the weighted edge $e$; see Figure \ref{fig:i-way-operation}(b). 
 
\fig{width=\linewidth}{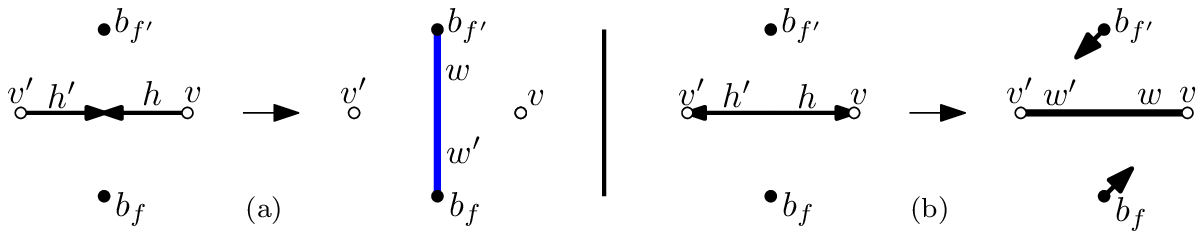}{Local transformation of a 0-way edge (left) and of a 2-way edge (right).}

With these definitions of local transformations for 0-way and 2-way edges, Definition \ref{def:master-bijections} of the mappings $\Phi_-$ and $\Phi_+$ is directly extended   to the case of weighted biorientations in $\mB$ and $\wB$ respectively. The mapping $\Phi_+$ is illustrated in Figure~\ref{fig:bijection_biorientation}. 
 
\fig{width=12cm}{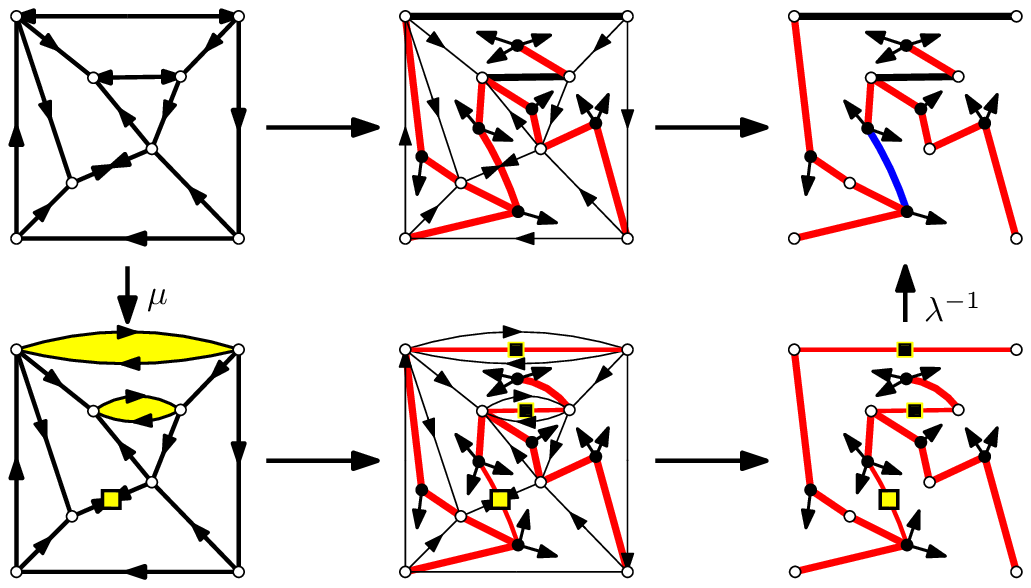}{Top: the bijection $\Phi_+$ applied to a biorientation in $\mB$ (weights are omitted). Bottom: how to reduce the case of biorientations to the case of ordinary orientations.}



\begin{theo}\label{thm:master-bijections-biorientation} 
Recall Definition~\ref{def:mBwB} for the sets $\mB$ and $\wB$ of weighted biorientations. 
\begin{itemize} 
\item 
 The mapping $\Phi_+$ is a bijection between $\mB$ and weighted mobiles of positive excess; the outer degree is mapped to the excess of the mobile. 
\item 
The mapping $\Phi_-$ is a bijection between $\wB$ and weighted mobiles of negative excess; the outer degree is mapped to minus the excess of the mobile. 
\end{itemize} 
\end{theo}

\begin{remark}\label{rk:parameter-correspondence} 
There are many obvious parameter-correspondences for the bijections~$\Phi_\pm$. 
For a start, the vertices, inner faces, 0-way edges, 1-way edges, 2-way edges in a biorientation $O$ are in natural bijection with the white vertices, black vertices, black-black edges, black-white edges, white-white edges of the mobile $\Phi_+(O)$. The indegree and weight of a vertex of $O$ are equal to the indegree and weight of the corresponding white vertex of  $\Phi_+(O)$. The degree, clockwise degree and weight of an inner face  of $O$ 
are equal to the degree, indegree and weight of the corresponding black vertex of $\Phi_+(O)$. 
 
Similar correspondences exist for $\Phi_-$, and we do not list them all as they follow directly from the definitions. 
\end{remark}

Before proving Theorem~\ref{thm:master-bijections-biorientation}, we state 
its consequences for $\NN$-biorientations. 
We call $\NN$\emph{-mobile} a weighted mobile with weights in $\NN$ that are positive at half-edges 
incident to white vertices and are zero at (non-bud) half-edges incident to black vertices 
(observe that $\NN$-mobiles with weight $1$ on each edge can be identified with unweighted properly bicolored mobiles). 

Specializing the master bijection $\Phi_\pm$ (Theorem~\ref{thm:master-bijections-biorientation}) to 
$\NN$-biorientations we obtain (the list of correspondences of parameters  
is restricted to what is needed later on): 
 
\begin{theo}\label{thm:kmaster-bijections} 
Recall Definition~\ref{def:mBwB} for the sets $\mB$ and $\wB$ of weighted biorientations. 
\begin{itemize} 
\item The mapping $\Phi_+$ is a bijection between $\NN$-biorientations in $\mB$ and $\NN$-mobiles of positive excess. 
For $T=\Phi_+(O)$, the outer degree, degrees of inner faces, weights of vertices, weights of edges  
in $O$ correspond respectively to the excess, degrees of black vertices, weights of white vertices, 
weights of edges in $T$. 
\item The mapping $\Phi_-$ is a bijection between $\NN$-biorientations in $\wB$ and $\NN$-mobiles of negative excess. 
For $T=\Phi_-(O)$, the outer degree, degrees of inner faces, weights of inner vertices, weights of inner edges  
in $O$ correspond respectively to the opposite of the excess, degrees of black vertices, weights of white vertices, 
weights of edges in $T$.  
\end{itemize} 
\end{theo}

The mappings  $\Phi_+$ and $\Phi_-$ are applied to some $\NN$-biorientations in Figure~\ref{fig:dual_kopening_ex}. 
 
\begin{figure} 
\begin{center} 
\includegraphics[width=\linewidth]{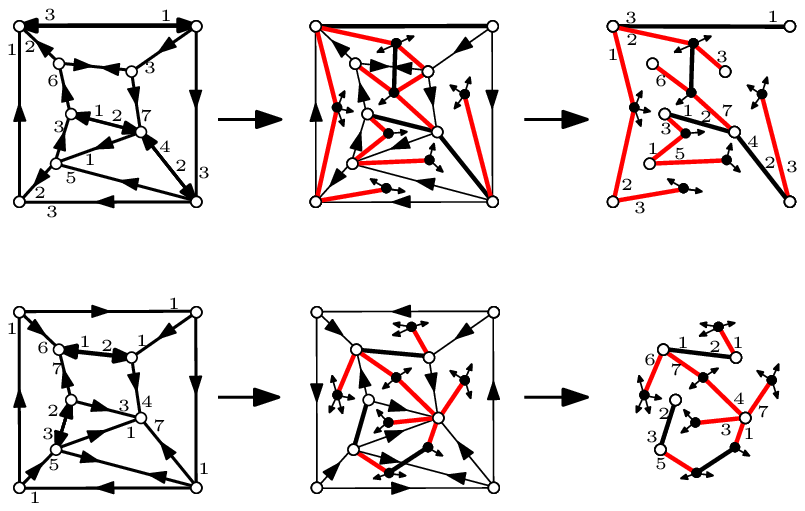} 
\end{center} 
\caption{Top: Master bijection $\Phi_+$ applied to an $\NN$-biorientation in $\mB$. 
Bottom: Master bijection $\Phi_-$  applied to an  $\NN$-biorientation in $\wB$. 
The vertex $b_{f_0}$ and the incident buds or edges are not represented. 
Weights (which are always zero) are not represented on outgoing half-edges in the orientations 
nor on half-edges incident to black vertices in the mobiles.} \label{fig:dual_kopening_ex} 
\end{figure} 
 

\begin{proof}[Proof of Theorem~\ref{thm:master-bijections-biorientation} assuming Theorem~\ref{thm:master-bijections} is proved] 
We show here that Theorem~\ref{thm:master-bijections-biorientation} (for 
weighted biorientations) follows from Theorem~\ref{thm:master-bijections} (for 
ordinary orientations). The idea of the reduction is illustrated in Figure~\ref{fig:bijection_biorientation}. 
First of all, observe that there is no need to consider weights in the proof. Indeed, through the mapping $\Phi_+$ applied to a weighted biorientation in $\mB$, the weight of each half-edge  is given to exactly one half-edge of the corresponding mobile, so that there is absolutely no weight constraint on the mobiles. Similarly, through the mapping $\Phi_-$ applied to a weighted biorientation in $\wB$, 
the weight of each inner half-edge is given 
to exactly one half-edge of the corresponding mobile, while the weights of outer half-edges  
 are fixed (outer edges are 1-way with weights $(0,1)$ in an admissible bi-orientation). 
 
We call \emph{bi-marked orientation} an orientation with some marked vertices of degree 2 and indegree 2  and some marked faces of degree 2 and clockwise degree 2. Observe that two marked vertices cannot be adjacent, two marked faces cannot be adjacent, and moreover a marked vertex cannot be incident to a marked face. Given an 
unweighted biorientation $O$ we denote by $\mu(O)$ the bi-marked orientation obtained by replacing every 0-way edge by two edges in series directed toward a marked vertex of degree 2 (and indegree 2), and replacing every 2-way edges by two edges in parallel directed clockwise around a marked face of degree 2 (and clockwise degree 2). The mapping $\mu$ illustrated in Figure~\ref{fig:bijection_biorientation} (left) is clearly a bijection between biorientations and bi-marked orientations. 
We call \emph{bi-marked mobile} a properly bicolored mobile with some set of non-adjacent marked vertices of degree 2, such that for each black-white edge $(b,w)$ with $w$ a marked white vertex of degree $2$, 
the next half-edge after $(b,w)$ in clockwise order around $b$ is a bud.  
Given a mobile $T$, we denote by $\lambda(T)$ the bi-marked mobile obtained by inserting a marked black vertex at the middle of each white-white edge, and inserting a marked white vertex $w$ at the middle of each black-black edge $e=(b,b')$ together with two buds: one in the corner following $e$ in clockwise order around $b$
and one in the corner following $e$ in clockwise order around $b'$. The mapping $\lambda$, which is illustrated in Figure~\ref{fig:bijection_biorientation} (right), is clearly a bijection between mobiles and bi-marked mobiles. 
Moreover it is clear from the definitions that the mapping $\Phi_{+}$ (resp. $\Phi_-$) for biorientations is equal to the mapping $\la\circ\Phi\circ\mu$ where the mapping $\Phi$ is the restriction of $\Phi_{+}$  (resp. $\Phi_-$) to ordinary orientations. Finally, note that the outer degree is the same in $O$ as in $\mu(O)$, 
and the excess is the same in $T$ as in $\lambda(T)$. 
 Thus Theorem~\ref{thm:master-bijections} implies that $\Phi_+$ (resp. $\Phi_-$) is a bijection on $\mB$ (resp. on $\wB$). 
\end{proof}


Before closing this section we state an additional claim which is useful for counting purposes. For any weighted biorientation $O$ in $\wB$, we call \emph{exposed} the buds of the mobile $T=\Phi_-(O)$ created by applying the local transformation to the outer edges of $O$ (which have preliminarily been returned). 
\begin{claim}\label{claim:exposed-corners} 
Let $O$ be a biorientation in $\wB$ and let $T=\Phi_-(O)$. There is a bijection between the set $\vec{O}$ of corner-rooted maps inducing the face-rooted map $O$ (i.e., the maps obtained by choosing a root-corner in the root-face of $O$), and the set  $\vec{T}$ of mobiles obtained from $T$ by marking an exposed bud.  Moreover, there is a bijection between the set $T_{\to}$ of mobiles obtained from $T$ by marking a non-exposed bud, and the set  $T_{\to\circ}$ of mobiles obtained from $T$ by marking a half-edge incident to a white vertex. 
\end{claim} 
We point out that there is a little subtlety in Claim~\ref{claim:exposed-corners} related to the possible symmetries (for instance, it could happen that the $d$ corners of the root-face of $O$ give less than $d$ different corner-rooted maps). 
 
\begin{proof} 
The natural bijection $\gamma$ between  the $d$ exposed buds of $T$ and the $d$ corners in the root-face of $O$ (in which an exposed bud $b$ points toward the vertex incident to the corner $\gamma(b)$) does not require any ``symmetry breaking'' convention. Hence the bijection $\gamma$ respects the possible symmetries of $O$ and $T$, and therefore it induces a bijection between  $\vec{T}$ and $\vec{O}$. We now prove the second bijection. Let $B$ be the set of non-exposed buds of $T$, and let $H$ be the set of half-edges incident to a white vertex. Using the fact that the local transformation of each inner $i$-way edge of $O$ gives $i$ buds in $B$ and $i$ half-edges in $H$ (and that all buds in $B$, and all half-edges in $H$ are obtained in this way), one can easily define a bijection $\gamma'$ between $B$ and $H$ without using any ``symmetry breaking'' convention. Since the bijection $\gamma'$ respects the symmetries of $T$ it induces  a bijection between  $T_{\to}$  and $T_{\to\circ}$. 
\end{proof}

 
\bigskip 
 
\section{Bijections for $d$-angulations of girth $d$}\label{sec:bij_dang} 
In this section $d$ denotes an integer such that $d\geq 3$.   
We prove the existence of a class of ``canonical'' $\NN$-biorientations for $d$-angulations of girth $d$. This allows us to identify the class $\mC_d$ of $d$-angulations of girth $d$ with a 
class of $\NN$-biorientations in $\wB$. We then obtain a bijection between $\mC_d$ and a class of 
$\NN$-mobiles by specializing the master bijection $\Phi_-$ to these orientations.\\


\subsection{Biorientations for $d$-angulations of girth $d$}~\\ 
Let $D$ be a face-rooted $d$-angulation. A \emph{\ddm-orientation} of $D$ is an $\NN$-biorientation 
such that each inner edge has weight $d-2$, each inner vertex has weight $d$, each outer edge 
has weight $1$ and each outer vertex has weight $1$. 
Thus, \ddm-orientations are $\alpha/\beta$-orientations where $\alpha(v)=d$ for inner vertices, $\alpha(v)=1$ for outer vertices, $\beta(e)=d-2$ for inner edges, and $\beta(e)=1$ for outer edges. 
 By Lemma~\ref{lem:unique_minimal}, if $D$ has a \ddm-orientation, then it has a unique minimal one.
 A minimal \ddm-orientation is represented in Figure~\ref{fig:bijection_dk} (top-left).

\begin{theo}\label{thm:ori} 
A face-rooted $d$-angulation admits a \ddm-orientation if and only if it has girth~$d$. In this case, the minimal \ddm-orientation is in $\wB$. 
\end{theo} 
 
\begin{proof} 
We first prove the necessity of having girth $d$. If $D$ has not girth $d$, then it is either a tree (its girth is infinite), or it has a cycle of length $c<d$. If $D=(V,E)$ is a tree, then $|E|=|V|-1$, 
all edges have weight $1$ and all vertices have weight $1$ (since all vertices and 
edges are incident to the outer face), which  
contradicts Condition~(i). We now suppose that $D$ has a simple cycle $C$ of length $c<d$. Let $D'$ be the submap of $D$ enclosed by $C$ which does not contain the root-face of $D$. Let $V',E'$ be respectively the set of vertices and edges strictly inside $C$. By Lemma~\ref{lem:counting}, $(d-2)|E'|=d|V'|+c-d<d|V'|$. This prevents the existence of a \ddm-orientation for $D$ because in such an orientation $(d-2)|E'|$ corresponds to the sum of weights of the edges in $E'$, while $d|V'|$ corresponds to the sum of weights of the vertices in $V'$. 
 
We now consider a $d$-angulation $D$ of girth $d$ and want to prove that $D$ has a \ddm-orientation. We use Lemma \ref{lem:exists-alpha} with the map $D=(V,E)$, the function $\alpha$ taking value $d$ on inner vertices and value $1$ on outer vertices, and the function $\beta$ taking value $d-2$ on inner edges and value $1$ on outer edges.   
We first check Condition (i). Observe that $D$ has $d$ distinct outer vertices (otherwise the outer face would create a cycle of length less than $d$). Hence $\sum_{v\in V}\alpha(v)=d(|V|-d)+d$. This is equal to $(d-2)(|E|-d)+d=\sum_{e\in E}\beta(e)$ by Lemma~\ref{lem:counting} (with $p=d$), hence Condition~(i) holds.  
We now check Conditions (ii) and (iii). It is in fact enough to check these conditions for \emph{connected} induced subgraphs (indeed, for  $S\subseteq V$,  
the quantities $\sum_{v\in S}\alpha(v)$ and $\sum_{e\in E_S}\beta(e)$ are additive over the connected components of $G_S=(S,E_S)$).   
Let $S$ be a subset of vertices such that the induced subgraph $G_S=(S,E_S)$ is connected.\\  
\textbf{Case (1)}: $G_S$ contains all outer vertices (and hence all outer edges) of $D$.  
Then $\sum_{v\in S}\alpha(v)=d\cdot(|S|-d)+d$, and $\sum_{e\in E_S}\beta(e)=(d-2)\cdot(|E_S|-d)+d$. Since 
$D$ has girth $d$, all faces of $G_S$ have degree at least $d$, so Lemma~\ref{lem:counting} gives $(d-2)\cdot(|E_S|-d)\leq d\cdot(|S|-d)$, that is, $\sum_{v\in S}\alpha(v)\geq  \sum_{e\in E_{S}}\beta(e)$.\\  
\textbf{Case (2)}: $S$ contains at least one outer vertex but not all outer vertices of $D$.  
Let $S'$ be the union of $S$ and of the set of outer vertices of $D$. Let $A$ be the set of outer vertices in $S'$ but not in $S$ 
and  let $B$ be the set of outer edges in $E_{S'}$ but not in $E_S$; note that $|B|>|A|$. By Case~(1), $\sum_{v\in S'}\alpha(v)\geq \sum_{e\in E_{S'}}\beta(e)$. Moreover,  $\sum_{v\in S}\alpha(v)=\sum_{v\in S'}\alpha(v)-|A|$ (because $A=S'\backslash S$) and $\sum_{e\in E_S}\beta(e)\leq \sum_{e\in E_{S'}}\beta(e)-|B|$ (because $B\subseteq E_{S'}\backslash E_S$). Hence $\sum_{v\in S}\alpha(v)>\sum_{e\in E_S}\beta(e)$.\\ 
\textbf{Case (3)}: $S$ contains no outer vertex. In that case $\sum_{v\in S}\alpha(v)=d\cdot |S|$ and $\sum_{e\in E_S}\beta(e)=(d-2)\cdot|E_S|$.  Since all faces in $G_S$ have degree at least $d$, Lemma~\ref{lem:counting} gives $(d-2)\cdot(|E_S|-d)\leq d\cdot(|S|-d)$. Thus, $(d-2)|E_S|<d|S|$,  that is, $\sum_{e\in E_S}\beta(e)<\sum_{v\in S}\alpha(v)$.\\  
To summarize, Condition~(ii) holds in all three cases. Moreover, the inequality $\sum_{v\in S}\alpha(v)\geq \sum_{e\in E_S}\beta(e)$ might be tight only in Case~(1), where all outer vertices are in $S$. Hence Condition~(iii) holds with respect to any outer vertex.  
 
We have proved that that any $d$-angulation $D$ of girth $d$ admits a \ddm-orientation, and that every
\ddm-orientation is accessible from any outer vertex. By Lemma~\ref{lem:unique_minimal} $D$ has a minimal \ddm-orientation, and by Remark~\ref{rk:minimal_Norientation} this orientation is admissible, hence clockwise-minimal. Thus, the minimal \ddm-orientation of $D$ is in~$\wB$.  
\end{proof}

\medskip 
 
\subsection{Specializing the master bijection $\Phi_-$}~\\ 
By Theorem~\ref{thm:ori}, the class $\mC_d$ of $d$-angulations of girth $d$ can be identified with the subset 
$\mE_d$ of $\NN$-biorientations in $\wB$ such that every face has degree $d$, every inner edge has weight $d-2$, 
and every inner vertex has weight $d$. We now characterize the mobiles that are in bijection with the subset $\mE_d$. 
 
We call \emph{$d$-branching mobile} an $\NN$-mobile such that every black vertex has degree $d$, 
every white vertex has weight $d$, and every edge has weight $d-2$. A 5-branching mobile is shown in Figure~\ref{fig:bijection_dk} (top-right). 
By Theorem~\ref{thm:kmaster-bijections}, the master bijection $\Phi_-$ induces a bijection between orientations in $\mE_d$ with $n$ inner faces and $d$-branching mobiles of excess $-d$ with $n$ black vertices. 
The additional condition of having excess $-d$ is in fact redundant (hence can be omitted) as claimed below. 
 
\begin{claim}\label{claim:dreg} 
Any $d$-branching mobile has excess $-d$. 
\end{claim} 
\begin{proof} 
Let $\ee,\ee',\oo$ be respectively the number of edges, white-white edges and buds of a $d$-branching mobile. 
 By definition, the excess is $\delta=\ee+\ee'-\oo$. Let $\bb$ and $\ww$ be respectively the number of black and white vertices. The degree condition for black vertices gives $\ee-\ee'+\oo=d\bb$, so that  $\delta=2\ee-d\bb$. 
 The weight condition on white vertices gives $(d-2)\ee=d\ww$ and combining this relation with $\bb+\ww=\ee+1$ (so as to eliminate $\ww$) gives $\bb=\frac{2}{d}\ee+1$, hence $\delta=-d$. 
\end{proof} 
 
From Claim~\ref{claim:dreg} and the preceding discussion we obtain: 
\begin{theo}\label{thm:bij_dangul} 
For $d\geq 3$ and $n\geq 1$, face-rooted $d$-angulations of girth $d$ with $n$ inner faces are in bijection with $d$-branching mobiles with $n$ black vertices. 
\end{theo}

\begin{figure} 
\begin{center} 
\includegraphics[width=\linewidth]{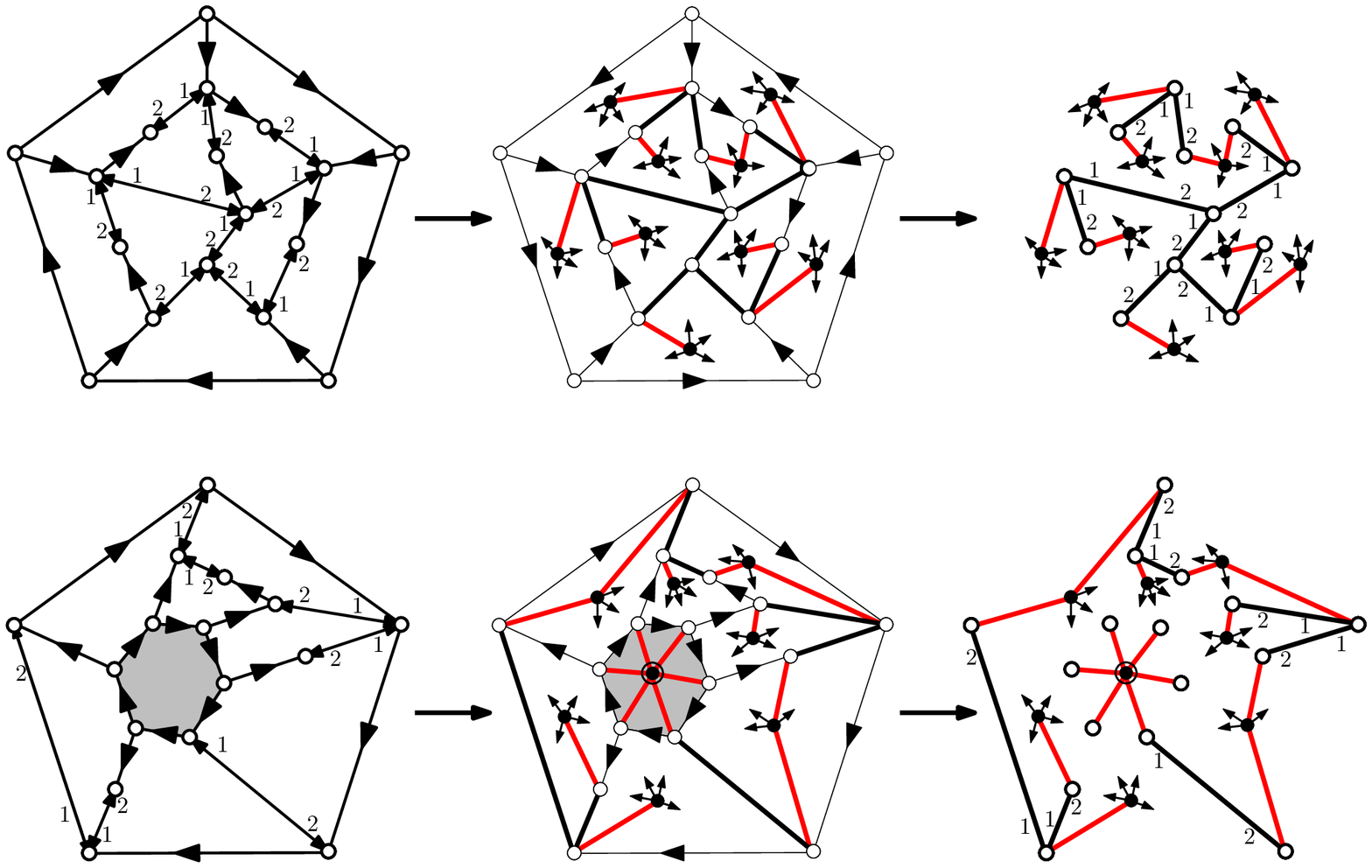} 
\end{center} 
\caption{Top: a 5-angulation $D$ of girth $5$ and the 5-branching mobile $\Phi_-(D)$. Bottom: a non-separated 6-annular 
5-angulation $A$ of girth $5$ and the (6,5)-branching mobile $\Phi_+(A)$. The weights are only indicated on 2-way edges.} 
\label{fig:bijection_dk} 
\end{figure}

Theorem~\ref{thm:bij_dangul} is illustrated in Figure~\ref{fig:bijection_dk} (top-part). Before closing this section we mention that a slight simplification appears in the definition of $d$-branching mobiles when $d$ is even. 
\begin{prop}\label{prop:bip} 
If $d=2b$ is even, then every half-edge of the $d$-branching mobile has an even weight (in $\{0,\ldots,d-2\}$).  Similarly the minimal \ddm-orientation of any $d$-angulation has only even weights.  
\end{prop} 
 
\begin{proof} 
Let $d$ be even and let $M$ be a $d$-branching mobile. 
Edges of $M$ have either two half-edges with even weights or two half-edges with odd weights, in which case we call them \emph{odd}. 
 We suppose by contradiction that the set of odd edges is non-empty.  In this case, there exists a vertex $v$ incident to exactly one odd edge (since the set of odd edges form a non-empty forest). Hence, the weight of $v$ is odd, which is impossible since the weight of black vertices is 0 and the weight of white vertices is $d$. 
 
Similarly, the edges of a \ddm-orientation have either two half-edges with even weights or two half-edges with odd weights, in which case we call them \emph{odd}. Moreover odd edges are 2-way hence they form a forest in the minimal \ddm-orientation (otherwise there would be a counterclockwise circuit). Hence, if there are odd edges, there is a vertex $v$ incident to only 1 odd edge, contradicting the requirement that every vertex has weight $d$. 
\end{proof} 
 
\begin{remark}\label{rk:2bangul} 
For $b\geq 2$, we call \emph{\bbm-orientation} an $\NN$-biorientation of a $2b$-angulation such that inner vertices have weight $b$, outer vertices have weight $1$, inner edges have weight $b-1$ and outer edges have weight $1$. We also call $b$-\emph{dibranching mobile} an $\NN$-mobile where black vertices have degree $2b$, white vertices have weight $b$, and edges have weight $b-1$.  
Then, in the case $d=2b$, Proposition~\ref{prop:bip} ensures that the bijection of Theorem~\ref{thm:bij_dangul} simplifies (dividing by $2$ the weights)  to a bijection between $2b$-angulations of girth $2b$ with $n$ inner faces and $b$-dibranching mobiles with $n$ black vertices. 
In particular, for $b=2$ (simple quadrangulations) the  $2/1$-orientations are ordinary orientations  with indegree $2$ at each inner vertex. These are precisely the orientations considered in~\cite{Fraysseix:Topological-aspect-orientations} and used for defining the bijection in~\cite[Sec.2.3.3]{Schaeffer:these}.
\end{remark}

\medskip 
 
\section{Bijections for $d$-angulations of girth $d$ with a $p$-gonal boundary}\label{sec:bij_dang_bond} 
In this section $p$ and $d$ denote integers satisfying $3\leq d\leq p$. 
We deal here with $d$-angulations with a  boundary. 
We call \emph{$p$-gonal $d$-angulation} a map having one marked face of degree $p$, called \emph{boundary face},  whose contour is simple (its vertices are all distinct) and all the other faces of degree $d$.  We call \emph{$p$-annular $d$-angulation} a face-rooted $p$-gonal $d$-angulation whose root-face is not the boundary face. Our goal is to obtain a bijection for $p$-annular $d$-angulations of girth $d$ by a method similar to the one of the previous section:  we first exhibit a canonical orientation for these maps and then consider the restriction of the master bijection to the canonical orientations. However, an additional difficulty arises: one needs to factorize $p$-annular $d$-angulations into two parts, one being a $d$-angulation 
of girth $d$ (without boundary) and the other being a \emph{non-separated} $p$-annular $d$-angulation of girth $d$ (see definitions below). The master bijections $\Phi_-$ and $\Phi_+$ then give bijections for each part.\\

\subsection{Biorientations for $p$-annular $d$-angulations}~\\
 Let $A$ be a $p$-annular $d$-angulation. The $p$ vertices incident to the boundary face are called \emph{boundary vertices}. A \emph{pseudo \ddm-orientation} of $A$ is an $\NN$-biorientation such that the edges have weight $d-2$, the non-boundary vertices have weight $d$,  and the boundary face is a clockwise circuit of 1-way edges. By Lemma~\ref{lem:counting}, the number $\vv$ of vertices and $\ee$ of edges of $A$ satisfy 
$(d-2)\ee=d(\vv-p)+dp-p-d$. Since $d(\vv-p)$ is the sum of the weights of the non-boundary vertices, this proves the following claim. 
\begin{claim}\label{claim:sum_in} 
The sum of the weights of the boundary vertices in a pseudo \ddm-orientation of a $p$-annular $d$-angulation is 
$dp-p-d$. 
\end{claim} 
 
We now prove that  pseudo \ddm-orientations characterize $p$-annular $d$-angulations of girth $d$. 
 
\begin{prop}\label{lem:exists_simple} 
A $p$-annular $d$-angulation $A$ admits a pseudo \ddm-orientation if and only if it has girth $d$. In this case, $A$ admits a unique minimal pseudo \ddm-orientation, and this orientation is accessible 
from any boundary vertex. 
\end{prop} 
 
\begin{proof} 
The necessity of having girth $d$ is proved in the same way as in Theorem~\ref{thm:ori}. We now consider a  $p$-annular $d$-angulation $A$ of girth $d$ and prove that it admits a pseudo \ddm-orientation. 
 
First we explain how to $d$-angulate the interior of the boundary face of $A$ while keeping the girth equal to $d$. Let $A'$ be any $p$-gonal $d$-angulation of girth $d$. We insert $A'$ inside the boundary face of $A$ (in such a way that the two cycles of length $p$ face each other). Then, if $d$ is even, we join each boundary vertex of $A$  to a distinct boundary vertex of $A'$ by a path of length $d/2-1$, while if $d$ is odd we join each boundary vertex of $A$  to two consecutive boundary vertices of $A'$ by a path of length $\lfloor d/2\rfloor$ (so that every boundary vertex of $A'$ is joined to two consecutive boundary vertices of $A$). It is easily seen that the resulting map is a $d$-angulation of girth $d$. 
 
By the preceding method, we obtain from $A$ a $d$-angulation $D$ of girth $d$. 
Moreover, one of the faces of $D$ created inside the boundary face $f_b$ shares no vertex (nor edge) with $f_b$. 
Take such a face as the root-face of $D$. 
By Theorem~\ref{thm:ori}, the face-rooted $d$-angulation $D$ admits a  \ddm-orientation $O$. The orientation $O$ induces an orientation $O_A$ of $A$ in which every non-boundary vertex has weight $d$. Let $O_A'$ be the orientation of $A$ obtained from $O_A$ by reorienting the edges incident to the boundary face $f_b$ into a clockwise circuit of 1-way edges. By definition, $O_A'$ is a pseudo \ddm-orientation. In addition, by Theorem \ref{thm:ori}, $O$ is accessible from any vertex incident to the root-face of $D$. Thus, for any vertex $v$ of $A$, there is a directed path $P$ of $O$ from a  vertex  of the root-face of $D$ to $v$. Since $P$ has to pass by the contour of the boundary face to reach $v$, there is a directed path of $O_A'$ from some boundary vertex to $v$. Since the boundary face of $O_A'$ is a circuit, it means that  there is a directed path of $O_A'$ from any boundary vertex to $v$. Thus, we have proved the existence of a pseudo \ddm-orientation $O_A'$ which is accessible from any boundary vertex. 
 
We now prove the statement about the minimal orientation (note that it is not a direct consequence of Lemma~\ref{lem:unique_minimal}). Let $\hA$ be the map obtained from $A$ by contracting the boundary face of $A$ into a vertex, denoted $v_b$ (the $p$ edges of the boundary face are contracted). The mapping $\sigma$ which associates to a pseudo \ddm-orientation of $A$ the induced orientation of $\hA$ is a bijection between the (non-empty) set $S$ of pseudo \ddm-orientations of  $A$  and the set $\hS$ of orientations of $\hA$ such that every vertex other than $v_b$ has weight $d$ (while $v_b$ has weight $p-d$ by Claim~\ref{claim:sum_in}). Moreover, an orientation $O\in S$ is accessible from all boundary vertices if and only if $\sigma(O)$ 
is accessible from $v_b$; and $O$ is minimal if and only if $\sigma(O)$ is minimal (indeed, the fact that the boundary face of $O$ is a circuit implies that any counterclockwise circuit of $\sigma(O)$ is part of a counterclockwise circuit of $O$). Lemma~\ref{lem:unique_minimal} ensures that there is a unique minimal biorientation in $\hS$, hence $A$ admits a unique minimal pseudo \ddm-orientation, denoted $O_{\mathrm{min}}$. Moreover Lemma~\ref{lem:exists-alpha} ensures that, if an orientation in $\hS$ is accessible from $v_b$, then all orientations in $\hS$ are accessible from $v_b$ (indeed the accessibility condition only depends on the weight-functions $\alpha$ and $\beta$). Since $\sigma(O_A')$ is accessible from $v_b$, we conclude that $\sigma(O_{\mathrm{min}})$ is 
also accessible from $v_b$, hence $O_{\mathrm{min}}$ is accessible from all boundary vertices of~$A$. 
\end{proof}

For a $p$-annular $d$-angulation, a simple cycle $C$ is said to be \emph{separating} if the boundary face and the root-face are on different sides of $C$ and if $C$ is not equal to the root-face contour. 
A $p$-annular $d$-angulation (of girth $d$) is \emph{non-separated} if it has no separating cycle of length $d$.

\begin{prop} \label{prop:charact} 
The minimal pseudo \ddm-orientation of a $p$-annular $d$-angulation $A$ of girth $d$ is clockwise-minimal accessible (i.e., in $\mB$) if and only if $A$ is non-separated. 
\end{prop}

\begin{proof} 
Let $A$ be a  $p$-annular $d$-angulation of girth $d$.
Suppose first that $A$ has a separating cycle $C$ of length $d$. Seeing the root-face as the outer face, the boundary face is inside $C$; and 
since $C$ is not the root-face contour and $A$ has girth $d$, it is easy to check that there is an outer vertex $u_0$ strictly outside $C$. 
We now prove that the minimal  pseudo \ddm-orientation $O$ is not accessible from $u_o$. Let $D$ be the $p$-gonal $d$-angulation made of $C$ and of all the edges and vertices inside $C$. By Lemma~\ref{lem:counting}, the numbers $\ee$ of edges and $\vv$ of vertices in $D$ satisfy $(d-2)\ee=d(\vv-p)+dp-p-d$. Moreover the right-hand-side is the sum of the weights of the vertices in $D$ (by Claim~\ref{claim:sum_in}). Thus all edges incident to $C$ in the region exterior to $C$ are 1-way edges oriented away from $C$. Thus $u_o$ cannot reach the vertices of $D$ by a directed path. The orientation $O$ is not accessible from the outer vertex $u_o$, hence $O$ is not clockwise-minimal accessible. 
 
Suppose conversely that $A$ is non-separated. We first prove that the minimal  pseudo \ddm-orientation $O$ is accessible from any outer vertex. We reason by contradiction and suppose that $O$ is not accessible from an outer vertex $u_o$. In this case, $u_o$ cannot reach the boundary vertices by a directed path (since Proposition~\ref{lem:exists_simple} ensures that the boundary vertices can reach all the other vertices). We consider the set $U$ of vertices of $A$ that can reach the boundary vertices by a directed path. This set contains all the boundary vertices but not the outer vertex $u_o$. 
Let $A_U=(U,E_U)$ be the map induced by the vertices in $U$ (it is clear from the definitions that $A_U$ is connected).  Every edge between a vertex $u$ in $U$ and a vertex not in $U$ is oriented away from $U$. Thus $(d-2)|E_U|$ equals the sum of the weights of the vertices in $U$. By Claim~\ref{claim:sum_in} we get $(d-2)|E_U|=d(|U|-p)+pd-p-d$ hence 
\begin{equation}\label{eq:Eu} 
(d-2)(|E_U|-p)=d(|U|-p)+p-d. 
\end{equation} 
Now, since $A$ has girth $d$, the non-boundary faces of $A_U$ have degree at least $d$. Therefore by Lemma~\ref{lem:counting}, the equality \eqref{eq:Eu} ensures that every face of $A_U$ has degree $d$. In particular, the cycle $C$ corresponding to the contour of the face of $A_U$ in which the outer vertex $u_o$ lies has length $d$. This cycle $C$ is separating, which is a contradiction. 
 
It remains to prove that the outer face contour of $O$ is a clockwise circuit. Suppose the contrary, i.e., there is an outer edge $e$ which is 1-way and has the root-face on its right. Let $v,v'$ be the end and origin of $e$. 
By accessibility from $v$ there is a simple directed path $P$ from $v$ to a boundary vertex $u$. 
By accessibility from boundary vertices, there is a simple directed path $P'$ from $u$ to $v'$. 
The paths $P$ and $P'$ do not contain the 1-way edge $e$. Thus, a circuit $C$ containing $e$ can be extracted from  the concatenation of $P$, $P'$ and $e$. Since $e$ has the root-face on its right, the circuit $C$ is counterclockwise, contradicting the minimality of $O$. 
\end{proof}

\medskip

\subsection{Specializing the master bijections $\Phi_{\pm}$}~\\ 
We call $(p,d)$\emph{-branching mobile} an $\NN$-mobile such that: 
\begin{itemize} 
\item every edge has weight $d-2$, 
\item every black vertex has degree $d$ except for one, called the \emph{special vertex} $s$, which has degree $p$ and is incident to no bud, 
\item every white vertex which is not a neighbor of $s$ has weight $d$, and the weights of the neighbors of $s$ add up to $pd-p-d$. 
\end{itemize} 
An example of $(6,5)$-branching mobile is represented in Figure~\ref{fig:bijection_dk} (bottom-right).\\ 
 
By Proposition~\ref{prop:charact}, the class of non-separated $p$-annular $d$-angulations of girth $d$ can be identified with the class $\mE_{p,d}$ of $\NN$-biorientations in $\mB$  such that 
\begin{itemize} 
\item every edge has weight $d-2$, 
\item every face has degree $d$ except for one non-root face of degree $p$, called boundary face, whose contour is a clockwise circuit of 1-way edges 
\item every non-boundary vertex has weight $d$ (and the sum of weights of the boundary vertices is $pd-p-d$ by Claim~\ref{claim:sum_in}). 
\end{itemize} 
By Theorem~\ref{thm:kmaster-bijections} (and the definition of $\Phi_+$, which implies that an inner face of degree $p$ whose contour is a clockwise circuit of 1-way edges corresponds to a black vertex of the mobile incident to $p$ edges and no bud), the  class $\mE_{p,d}$ of orientations is in bijection with the class of $(p,d)$-branching mobiles of excess $d$. As in the case of $d$-angulations without boundary, the additional requirement on the excess is redundant as claimed below. 
 
\begin{claim} 
Any $(p,d)$-branching mobile has excess $d$. 
\end{claim} 
\begin{proof} 
By definition, the excess is $\delta=\ee+\ee'-\oo$, where  $\ee,\ee',\oo$ are respectively the number of edges, white-white edges and buds. Let $\bb$ and $\ww$ be the number of black and white vertices. 
The condition on black vertices gives $\ee-\ee'+\oo=d(\bb-1)+p$, so that  $\delta=2\ee-d\bb+d-p$. 
The condition on white vertices gives $(d-2)\ee=d(\ww-p)+pd-p-d$ and combining this relation with $\bb+\ww=\ee+1$ (so as to eliminate $\ww$) gives $\bb=\frac{2}{d}\ee-\frac{p}{d}$, hence $\delta=d$. 
\end{proof}

To summarize, by specializing the master bijection $\Phi_+$, Proposition~\ref{prop:charact} gives the following result. 
 
\begin{theo}\label{thm:bijdk} 
For $n\geq 1$, non-separated $p$-annular $d$-angulations of girth $d$  with $n$ inner faces are in bijection with $(p,d)$-branching mobiles with $n$ black vertices. 
\end{theo}

Theorem~\ref{thm:bijdk} is illustrated in Figure~\ref{fig:bijection_dk} (bottom).  As in the case without boundary, a slight simplification appears in the definition of $(p,d)$-branching mobiles when $d$ is even. First observe that when $d$ is even the  $p$-gonal $d$-angulations are bipartite (since the non-boundary faces generate all cycles). Thus, there is no $p$-gonal $d$-angulation when $p$ is odd and $d$ even. Similarly there is no $(p,d)$-branching mobile with $p$ odd and $d$ even. There is a further simplification (yielding a remark 
similar to Remark~\ref{rk:2bangul}):  
\begin{prop}\label{prop:pbip} 
If $d=2b$ is even, then for any $q\geq b$, the weights of half-edges of $(2q,2b)$-branching mobiles are even. 
Similarly,  the weights of the half-edges in minimal pseudo \ddm-orientations are even. 
\end{prop} 
 
\begin{proof} The proof is the same as the proof of Proposition~\ref{prop:bip}.\end{proof}

We now explain how to deal bijectively with general (possibly separated) $p$-annular $d$-angulations of girth $d$. Let $A$ be a $p$-annular $d$-angulation. A simple cycle $C$ of $A$ is called \emph{pseudo-separating} if it is either separating or is the contour of the root-face. A pseudo-separating cycle $C$ of $A$ of length $d$ defines: 
\begin{itemize} 
\item a $p$-annular $d$-angulation, denoted $A_C$, corresponding to the map on the side of $C$ containing the boundary face ($C$ is the contour of the root-face), 
\item a $d$-annular $d$-angulation, denoted $A_C'$, corresponding to the map on the side of $C$ containing the root-face  ($C$ is the contour of the boundary face). 
\end{itemize} 
In order to make the decomposition $(A,C)\mapsto (A_C,A_C')$ injective we consider \emph{marked maps}. We call a $p$-annular $d$-angulation \emph{marked} if a boundary vertex is marked.  
We now consider an arbitrary convention that, for each marked $p$-annular $d$-angulation, distinguishes one of the outer vertices as \emph{co-marked} (the co-marked vertex is entirely determined by the marked vertex). For a marked $p$-annular $d$-angulation $A^\bu$ and a pseudo-separating cycle $C$ of length $d$, we define the marked annular $d$-angulations $A_C^\bu$ and ${A_{C}^\bu}'$ obtained by marking $A_C$ at the marked vertex of $A^\bu$ and marking ${A_C^\bu}'$ at the co-marked vertex of $A_C^\bu$.

\begin{prop}\label{prop:decomposition-pannular} 
Any $p$-annular $d$-angulation $A$ of girth $d$ has a unique pseudo-separating cycle $C$ of length $d$ such that $A_C$ is non-separated. Moreover, the mapping $\Delta$ which associates to a marked  $p$-annular $d$-angulation  $A^\bu$ of girth $d$ the pair $(A_C^\bu,{A^\bu_C}')$ is a bijection between marked $p$-annular $d$-angulations of girth $d$ and pairs made of a marked non-separated $p$-annular $d$-angulation of girth $d$ and a marked $d$-annular $d$-angulation of girth $d$. 
\end{prop} 
 
Proposition~\ref{prop:decomposition-pannular} gives a bijective approach for general $p$-annular $d$-angulations of girth $d$: apply the master bijection $\Phi_+$ (Theorem~\ref{thm:bijdk}) to the non-separated $p$-annular $d$-angulation $A_C$, and the master bijection $\Phi_-$ (Theorem~\ref{thm:bij_dangul}) to the $d$-angulation~$A_C'$. 
 
\begin{proof} 
The set  $\mC$ of pseudo-separating cycles of $A$ of length $d$ is non-empty (it contains the cycle corresponding to the contour of the root-face) and is partially ordered by the relation $\preceq $ defined by setting $C\preceq C'$ if $A_C\subseteq A_{C'}$ (here the inclusion is in terms of the vertex set and the edge set). Moreover, $A_C$ is non-separated if and only if $C$ is a minimal element of $(\mC,\preceq)$. Observe also that for all $C,C'\in \mC$, the intersection $A_C \cap A_{C'}$ is non-empty because it contains the boundary vertices and edges.  Thus, if $C,C'$ are not comparable for  $\preceq$, then the cycles $C,C'$ must intersect. Since they have both length $d$ and $A$ has no cycle of length less than $d$, the only possibility is that  $A_C \cap A_{C'}$  is delimited by a separating cycle $C''$ of length $d$: $A_C \cap A_{C'}=A_{C''}$. This shows that  $(\mC,\preceq)$ has a unique minimal element, hence $A$ has a unique pseudo-separating cycle $C$ such that $A_C$ is non-separated. 
 
We now prove the second statement. The injectivity of $\Delta$ is clear: glue back the maps $A_C^\bu,{A^\bu_C}'$ by identifying the root-face $C$ of $A_C^\bu$ with the boundary face of ${A_C^\bu}'$, and identifying the co-marked vertex of $A^\bu_C$ with the marked vertex of ${A^\bu_C}'$. To prove surjectivity we only need to observe that gluing a marked non-separated $p$-annular $d$-angulation of girth $d$ and a marked $d$-annular $d$-angulation of girth $d$ as described above preserves the girth $d$ (hence the glued map $A^\bu$ is a  $p$-annular $d$-angulation) and makes $C$ the minimal (for $\preceq$) pseudo-separating cycle of length $d$ of $A^\bu$. 
\end{proof}

\medskip 
 
\section{Counting results} \label{sec:counting} 
In this section, we establish equations for the generating functions of $d$-angulations of girth $d$ without and with boundary (Subsections~\ref{subsec:counting-without} and~\ref{subsec:counting-with}). We then obtain closed formulas in the cases $d=3$ and $d=4$. 
We call \emph{generating function}, or \emph{GF} for short, of a class $C$ \emph{counted according to a parameter}  $P$ the formal power series $C(x)=\sum_{n\in\NN}c_n x^n$, where $c_n$ is the number of objects $c\in C$ satisfying $P(c)=n$ (we say that $x$ \emph{marks} the parameter $P$).  We also denote by $[x^n]G(x)$ the coefficient of $x^n$ in a formal power series $G(x)$.

\subsection{Counting rooted $d$-angulations of girth $d$}\label{subsec:counting-without}~\\ 
Let $d\geq 3$.  
By Theorem~\ref{thm:bij_dangul}, counting $d$-angulations of girth $d$ reduces to counting $d$-branching mobiles. We will characterize the generating function of $d$-branching mobiles by a system of equations (obtained by a straightforward recursive decomposition). We first need a few notations. For a positive integer $j$ we define the polynomial $h_j$ in the variables $w_1,w_2,\ldots$ by:
\begin{equation}\label{eq:def_hj} 
h_j(w_1,w_2,\ldots):=[t^j]\frac{1}{1-\sum_{i>0}t^i w_i}=\sum_{r>0}\sum_{\substack{i_1,\ldots,i_r>0\\ i_1+\ldots+i_r=j}}w_{i_1}\cdots w_{i_r}. 
\end{equation} 
In other words, $h_j$ is the (polynomial) generating function of integer compositions of $j$ 
where the variable $w_i$ marks the number of parts of size $i$.

We call \emph{planted $d$-branching mobile} an $\NN$-mobile with a marked leaf (vertex of degree 1) such that 
edges have weight $d-2$, non-marked black vertices have degree $d$ and non-marked white vertices have weight $d$. For $i\in\{0,\ldots,d-2\}$, we denote by $\cW_i$ the family of planted $d$-branching 
mobiles where the marked leaf has weight $i$ (the case $i=0$ correspond to a black marked leaf). We also denote by $W_i\equiv W_i(x)$ the generating function of $\cW_i$ counted according to the number of non-marked black vertices. 
 
For $i=d-2$, the marked leaf of a mobile in $\cW_i$ is a white vertex connected to a black vertex $b$. The $d-1$ other half-edges incident to $b$ are either buds or belong to an edge $e$ leading to white vertex $w$. In the second case the sub-mobile planted at $b$ and containing $w$ belongs to $\cW_0$. This decomposition gives the equation $W_{d-2}=x(1+W_0)^{d-1}$. 
 
For $i\in\{0,\ldots,d-3\}$, the marked leaf of a mobile in $\cW_i$ is connected to a white vertex $w$. Let $v_1,\ldots,v_r$ be the other neighbors of $w$. For $j=1\ldots, r$ the sub-mobile planted at $w$ and containing $v_j$ belongs to one of the classes $\cW_{\al(j)}$ (with $\al(j)>0$), and the sum of the indices $\sum_{j=1}^r \al(j)$ is equal to $d-(d-2-i)=i+2$. This decomposition gives 
$$ 
 W_i=\sum_{r>0}\sum_{\substack{i_1,\ldots,i_r>0\\ i_1+\ldots+i_r=i+2}}W_{i_1}\cdots W_{i_r}=h_{i+2}(W_1,\ldots,W_{d-2}). 
 $$ 
Hence, the series $W_0,W_1,\ldots,W_{d-2}$ satisfy the following system of equations: 
\begin{equation}\label{eq:syst1} 
\left\{ 
\begin{array}{rcl}\ds 
W_{d-2}&=&x(1+W_0)^{d-1}, \\ 
W_i&=&h_{i+2}(W_1,\ldots,W_{d-2})\ \ \ \forall i\in\{0,\ldots, d-3\},\\ 
W_i&=&0  \ \ \ \forall i>d-2. 
\end{array} 
\right. 
\end{equation} 
Observe that the system \eqref{eq:syst1} determines  $W_0,W_1,\ldots,W_{d-2}$ uniquely as formal power series. Indeed, it is clear that any solutions $W_0,W_1,\ldots,W_{d-2}$ of this system have zero constant coefficient. And from this observation it is clear that the other coefficients are uniquely determined by induction.

\smallskip 
 
\noindent The following table shows the system for the first values, $d=3,4,5$:

\vspace{.2cm} 
\centerline{\begin{tabular}{|l|l|l|} 
\hline 
$d=3$ & $d=4$ & $d=5$\\ 
\hline 
$W_1=x(1+W_0)^2$ & $W_2=x(1+W_0)^{3}\rule{0pt}{2.4ex}$ 
& $W_3=x(1+W_0)^4$\\ 
$W_0=W_1^2$ & $W_0=W_1^2+W_2$ & $W_0=W_1^2+W_2$\\ 
 & $W_1=W_1^3+2W_1W_2$ & $W_1=W_1^3+2W_1W_2+W_3$\\ 
 & & $W_2=W_1^4+3W_1^2W_2+2W_1W_3+W_2^2$\\[.1cm] 
 \hline 
\end{tabular} 
} 
\vspace{.2cm}

In the case where $d$ is even, $d=2b$, one easily checks that $W_i=0$ for $i$ odd (this property is related to Proposition~\ref{prop:bip} and follows from the fact that, for odd $i$, all monomials in $h_{i+2}(W_1,\ldots,W_{d-2})$ contain a $W_j$ with odd $j$).  
Hence the system can be simplified. The series $V_i:=W_{2i}$ satisfy the system: 
\begin{equation}\label{eq:syst2} 
\left\{ 
\begin{array}{rcl}\ds 
V_{b-1}&=&x(1+V_0)^{2b-1},\\ 
V_i&=&h_{i+1}(V_1,\ldots,V_{b-1})\ \ \ \forall i\in\{0,\ldots,b-2\},\\ 
V_i&=&0  \ \ \ \forall i>b-1. 
\end{array} 
\right. 
\end{equation} 
Equivalently,~\eqref{eq:syst2} is obtained by a decomposition strategy for $b$-dibranching mobiles (defined in Remark~\ref{rk:2bangul}) very similar to the one for $d$-branching mobiles.  
 
\smallskip 
 
\noindent The following table shows the system for the first values, $b=2,3,4$: 
 
\vspace{.2cm} 
 
\centerline{ 
\begin{tabular}{|l|l|l|} 
\hline 
$b=2$ & $b=3$ & $b=4$\\ 
\hline 
$V_1=x(1+V_0)^3$ & $V_2=x(1+V_0)^{5} \rule{0pt}{2.4ex}$ 
& $V_3=x(1+V_0)^7$\\ 
$V_0=V_1$ & $V_0=V_1$ & $V_0=V_1$\\ 
 & $V_1=V_1^2+V_2$ & $V_1=V_1^2+V_2$\\ 
 & & $V_2=V_1^3+2V_1V_2+V_3$\\[.1cm] 
 \hline 
\end{tabular} 
} 
\vspace{.2cm}

Let $\cM_d$ be the family of $d$-branching mobiles rooted at a corner 
incident to a black vertex, and 
let $M_d(x)$ be the GF of $\cM_d$ counted according to the number of non-root black vertices. 
Since each of the $d$ half-edges incident to the root-vertex is either a bud or is 
connected to a planted mobile in $\cW_0$ we have 
$$ 
M_d(x)=(1+W_0)^{d}. 
$$

\begin{prop}[Counting $d$-angulations of girth $d$]\label{prop:countd} 
For $d\geq 3$, the generating function $F_d(x)$ of corner-rooted $d$-angulations of girth $d$ counted according to the number of inner faces has the following expression: 
$$F_{d}=W_{d-2}-\sum_{i=0}^{d-3}W_i W_{d-2-i},$$ 
where the series $W_0,\ldots, W_{d-2}$ are the unique power series solutions of the system~\eqref{eq:syst1}. Therefore $F_d$ is algebraic. Moreover it satisfies 
$$ 
F_d'(x)=(1+W_0)^d. 
$$ 
In the bipartite case, $d=2b$, the series expressions simplify to 
$$ 
F_{2b}=V_{b-1}-\sum_{i=0}^{b-2}V_i V_{b-1-i},\ \ \ F_{2b}'(x)=(1+V_0)^{2b}, 
$$ 
where the series $V_0,\ldots,V_{b-1}$ are the unique power series solutions of the system~\eqref{eq:syst2}.

\end{prop} 
\begin{proof} 
Theorem~\ref{thm:bij_dangul} and Claim~\ref{claim:exposed-corners} (first part) imply that $F_{d}$ is the series counting $d$-branching mobiles rooted at an exposed bud. Call $B$ (resp. $H$) the GF of $d$-branching mobiles (counted according to the number of black vertices) with a marked bud (resp. marked half-edge incident to a white vertex). The second part of Claim~\ref{claim:exposed-corners} yields $F_{d}=B-H$. A mobile with a marked bud identifies to a planted mobile in $\cW_{d-2}$ (planted mobile where the vertex connected to the marked leaf is black, so the marked leaf can be taken as root-bud), hence $B=W_{d-2}$. 
A mobile with a marked half-edge incident to a white vertex identifies with an ordered pair of planted mobiles $(w,w')$ in  $\cW_i\times \cW_{d-2-i}$ for some $i$ in $\{0,\ldots,d-3\}$, hence $H=\sum_{i=0}^{d-3}W_iW_{d-2-i}$. 
 
Concerning the expression of $F_d'(x)$, observe that $F_d'(x)$ is the GF of corner-rooted $d$-angulations 
of girth $d$ with a secondary marked inner face. Equivalently, $F_d'(x)$ counts face-rooted $d$-angulations of girth $d$ 
with a secondary marked corner not incident to the root-face. By the master bijection $\Phi_-$, marking a corner not incident to the root-face in the $d$-angulation is equivalent to marking a corner incident to a black vertex in the corresponding mobile. Hence, Theorem~\ref{thm:bij_dangul} gives $F_d'(x)=M_d(x)=(1+W_0)^d$. 
\end{proof} 
 
The cases $d=3$ and $d=4$ correspond to simple triangulations and 
simple quadrangulations; we recover (see Section~\ref{sec:exact-counting}) exact-counting formulas 
due to Brown~\cite{Brown:triang3connexes+boundary,Brown:quadrangulation+boundary}. 
For $d\geq 5$ the counting results are new (to the best of our knowledge). For $d=5$ and $d=6$, one gets $F_5=x+5x^3+121x^5+4690x^7+228065x^9+O(x^{11})$, 
and $F_6=x+3x^2+17x^3+128x^4+1131x^5+11070x^6+O(x^7)$.

For any $d\geq 3$ a simple analysis based on the Drmota-Lalley-Wood theorem~\cite[VII.6]{fla} 
ensures that for odd $d$, the coefficient $[x^{2n}]F_d'(x)$ (from the Euler relation the 
odd coefficients of $F_d'(x)$ are zero) 
is asymptotically $c_d\,{\gamma_d}^n\,n^{-3/2}$ for some computable  positive 
constants $c_d$ and $\gamma_d$ depending on $d$. 
For even $d$, the coefficient $[x^{n}]F_d'(x)$ is asymptotically 
$c_d\,{\gamma_d}^n\,n^{-3/2}$ again with $c_d$ and $\gamma_d$ computable 
constants. Since $[x^n]F_d(x)=\tfrac1{n}[x^{n-1}]F_d'(x)$, the 
number of corner-rooted $d$-angulations of girth $d$ follows (up to the 
parity requirement for odd $d$) the asymptotic behavior $c_d\,{\gamma_d}^n\,n^{-5/2}$ which is universal for families of rooted planar maps.

\subsection{Counting rooted $p$-gonal $d$-angulations of girth $d$}\label{subsec:counting-with}~\\ 
Let $p\geq d\geq 3$. We start by characterizing the generating function of $(p,d)$-branching mobiles. A $(p,d)$-branching mobile is said  to be \emph{marked} if one of the $p$ corners incident to the special vertex is marked. Let $\cM_{p,d}$ be the class of marked  $(p,d)$-branching mobiles, and let $M_{p,d}(x)$  be the GF of this class counted according to the number of non-special black vertices. 
Given a mobile in $\cM_{p,d}$, we consider the (white) neighbors $v_1,\ldots,v_p$ of the special vertex $s$. 
For $i=1\ldots p$ the sub-mobiles $M_1,\ldots,M_{r_i}$ planted at $v_i$ (and not containing $s$) belong to some classes $\cW_{\al(i,1)},\ldots,\cW_{\al(i,r_i)}$ (with $\alpha(i,j)>0$) where $d-2+\sum_{j=1}^{r_i}\al(i,j)$ is the weight of $v_i$. Moreover, by definition of $(p,d)$-branching mobiles, the sum of the weights of  $v_1,\ldots,v_p$ is $pd-p-d=(d-2)p+p-d$. This decomposition gives 
\begin{equation}\label{eq:Mk} 
M_{p,d}(x)=h_{p-d}^{(p)}(W_1,\ldots,W_{d-2}), 
\end{equation} 
where for $j\geq 0$, $h_j^{(p)}\equiv h_j^{(p)}(w_1,w_2,\ldots)$ is the generating function of all $p$-tuples of compositions of non-negative integers $i_1,\ldots,i_p$ such that $i_1+\ldots+i_p=j$, where $w_i$ marks the number of parts of size $i$. It is clear that  $h_j^{(1)}=h_{j}$ and that $h_j^{(p)}$ satisfies: 
\begin{equation}\label{eq:defq} 
h_j^{(p)}(w_1,w_2,\ldots):=[t^j]\frac{1}{(1-\sum_{i>0}t^i w_i)^p}. 
\end{equation} 
Observe that in the special case $p=d$, there is a unique $(d,d)$-branching mobile, hence $M_{d,d}(x)=1$ (which is coherent with \eqref{eq:Mk} since $h_0^{(p)}=1$). 
 
We now turn to $p$-annular $d$-angulations. We denote by $\cA_{p,d}$ the family of \emph{marked} $p$-annular $d$-angulations of girth $d$ (a boundary vertex is marked), and we denote by $\cB_{p,d}$ the subfamily of those that are non-separated. Let $A_{p,d}(x)$ and $B_{p,d}(x)$ be the GF of these classes counted according to the number of non-boundary inner faces. By  Theorem~\ref{thm:bijdk}, the marked  non-separated  $p$-annular $d$-angulations of girth $d$ are in bijection by $\Phi_+$ with the marked $(p,d)$-branching mobiles (since marking the mobile is equivalent to marking the $p$-annular $d$-angulation at a boundary vertex). Thus, $B_{p,d}(x)=M_{p,d}(x)$. Moreover, Proposition~\ref{prop:decomposition-pannular} directly implies: 
$$ 
A_{p,d}(x)=M_{p,d}(x)F_d'(x), 
$$ 
since the series $F_d'(x)$ (defined in Proposition~\ref{prop:countd}) counts marked $d$-annular $d$-angulations of girth $d$.  We summarize:

\begin{prop}[Counting $p$-gonal $d$-angulations of girth $d$]\label{prop:countdk} 
For $p\geq d\geq 3$, the generating function $F_{p,d}(x)$ of $p$-gonal $d$-angulations of girth $d$ rooted at a corner 
incident to the boundary face, and counted according to the number of 
non-boundary faces satisfies 
$$ 
F'_{p,d}(x)=h_{p-d}^{(p)}(W_1,\ldots,W_{d-2})\cdot(1+W_0)^d, 
$$ 
where $h_j^{(p)}$ is defined in~\eqref{eq:defq} and the series $W_1,\ldots, W_{d-2}$ 
are specified by~\eqref{eq:syst1}. 
 
In the bipartite case, $d=2b$ and $p=2q$, the expression simplifies to 
$$ 
F'_{2q,2b}(x)=h_{q-b}^{(2q)}(V_1,\ldots,V_{b-1})\cdot(1+V_0)^{2b}, 
$$ 
where the  series $V_1,\ldots,V_{b-1}$ are specified by~\eqref{eq:syst2}. 
\end{prop} 
\begin{proof} 
By definition, a $p$-annular $d$-angulation is a $p$-gonal $d$-angulation  
with a secondary marked face of degree $d$. Hence, 
 $F'_{p,d}(x)=A_{p,d}(x)=M_{p,d}(x)F_d'(x)$. 
Moreover $F'_d(x)=(1+W_0)^d$ by Proposition~\ref{prop:countd} and  $M_{p,d}(x)=h_{p-d}^{(p)}(W_1,\ldots,W_{d-2})$ by \eqref{eq:Mk}. 
\end{proof}

\subsection{Exact-counting formulas for triangulations and quadrangulations}\label{sec:exact-counting}~\\ 
In the case of triangulations and quadrangulations  ($d=3,4$) the 
system of equations given by Proposition~\ref{prop:countdk} 
takes a form amenable to the Lagrange inversion formula. 
We thus recover bijectively the counting formulas established by Brown~\cite{Brown:triang3connexes+boundary,Brown:quadrangulation+boundary} 
for simple triangulations and simple quadrangulations with a boundary.

\begin{prop}[Counting simple $p$-gonal triangulations] 
For $p\geq 3$, $n\geq 0$, let $t_{p,n}$  be the number of corner-rooted $p$-gonal triangulations with $n+p$ vertices which are simple (i.e., have girth 3) and have the root-corner in the $p$-gonal face (no restriction on the root-corner for $p=3$). 
The generating function $T_p(x)=\sum_{n\geq 0}(2n+p-2)\,t_{p,n}x^n$ satisfies 
$$ 
T_p(x):=\binom{2p-4}{p-3}\,u^{2p-3},\ \ \ \mathrm{ where }\ \ u=1+x\,u^4. 
$$ 
Consequently, the Lagrange inversion formula gives: 
$$ 
t_{p,n}=\frac{2(2p-3)!}{(p-1)!(p-3)!}\,\frac{(4n+2p-5)!}{n!(3n+2p-3)!}. 
$$ 
\end{prop} 
\begin{proof} 
By the Euler relation, a $p$-gonal triangulation with $n+p$ vertices has $2n+p-1$ faces. 
Hence $T_p(x^2)x^{p-3}=F'_{p,3}(x)$. 
By Proposition~\ref{prop:countdk}, 
$$ 
F'_{p,3}(x)=\binom{2p-4}{p-3}W_1^{p-3}(1+W_0)^3=x^{p-3}\binom{2p-4}{p-3}(1+W_0)^{2p-3}, 
$$ 
where $W_0,W_1$ are specified by $\{W_1=x(1+W_0)^2,\ W_0=W_1^2\}$. Thus, $T_p(x^2)=\binom{2p-4}{p-3}(1+W_0)^{2p-3}$. 
Moreover, $1+W_0=1+W_1^2=1+x^2(1+W_0)^4$, hence $1+W_0(x)=u(x^2)$, where $u\equiv u(x)$ is the series specified by $u=1+xu^4$. Thus, $T_p(x)=\binom{2p-4}{p-3}\,u^{2p-3}$. 
\end{proof}

\begin{prop}[Counting simple $2p$-gonal quadrangulations] 
For $p\geq 2$, $n\geq 0$, let $q_{p,n}$ be the number of corner-rooted $2p$-gonal quadrangulations with $n+2p$ vertices which are simple (i.e., have girth 4) and have the root-corner incident to the $2p$-gonal face (no restriction on the root-corner for $p=2$). 
The generating function $Q_{p}(x)=\sum_{n\geq 0}(n+p-1)q_{p,n}x^n$ satisfies 
\beq 
Q_{p}(x)=\binom{3p-3}{p-2}\,u^{3p-2},\ \ \ \mathrm{where}\ \ u=1+x\,u^3. 
\eeq 
Consequently, the Lagrange inversion formula gives: 
$$ 
~~q_{p,n}=\frac{3(3p-2)!}{(p-2)!(2p-1)!}\,\frac{(3n+3p-4)!}{n!(2n+3p-2)!}. 
$$ 
\end{prop} 
\begin{proof} 
By the Euler relation, a $2p$-gonal quadrangulation with $n+2p$ vertices has $n+p$ faces. 
Hence $Q_p(x)x^{p-2}=F'_{2p,4}(x)$. 
By Proposition~\ref{prop:countdk}, 
$$ 
F'_{2p,4}(x)=\binom{3p-3}{p-2}V_1^{p-2}(1+V_0)^4=x^{p-2}\binom{3p-3}{p-2}(1+V_0)^{3p-2}, 
$$ 
where $V_0,V_1$ are specified by $\{V_1=x(1+V_0)^3,\ V_0=V_1\}$. Thus, $Q_p(x)=\binom{3p-3}{p-2}(1+V_0)^{3p-2}$. 
Moreover, $1+V_0=1+V_1=1+x(1+V_0)^3$, hence $1+V_0(x)=u(x)$, where $u\equiv u(x)$ 
is the series specified by $u=1+xu^3$. Thus, $Q_p(x)=\binom{3p-3}{p-2}\,u^{3p-2}$. 
\end{proof}


\bigskip

\section{Proof that the mappings $\Phi_+$, $\Phi_-$ are bijections}\label{sec:fromOBtoPhi} 
In this section, we prove Theorem~\ref{thm:master-bijections} thanks to a reduction to a bijection $\Psi$ described in~\cite{OB:boisees}. In this section, \emph{the orientations are non-weighted, and the mobiles are properly bicolored}. 
The relation between $\Psi$ and the master bijections $\Phi_+$, $\Phi_-$ involves \emph{duality}.\\

\titre{Duality.} 
 The \emph{dual} $O^*$ of an orientation $O$ is obtained by the process represented in Figure~\ref{fig:duality}(a): 
\begin{itemize} 
\item place a vertex $v_f$ of $O^*$ in each face $f$ of $O$, 
\item for each edge $e$ of $O$ having a face $f$ on its left and $f'$ on its right, draw a \emph{dual edge} $e^*$ of $O^*$ oriented from $v_f$ to  $v_{f'}$ across $e$. 
\end{itemize} 
Note that the duality (as defined above) is not an involution (applying duality twice returns every edge). Duality maps vertex-rooted orientations to face-rooted orientations. It is easy to check that a face-rooted orientation $O$ is minimal if and only if $O^*$ is \emph{accessible}, that is, accessible from the root-vertex. Similarly, a vertex-rooted orientation $O$ is accessible if and only if $O^*$ is \emph{maximal}, that is, has no clockwise circuit. 
For each $d\geq 1$, define $\mO_d$ as the set of clockwise-minimal accessible orientations of outer degree $d$ (so $\mO=\cup_{d\geq 1}\mO_d$),  
and define $\wO_d$ as the subset of orientations in $\mO_d$ that are admissible (so $\wO=\cup_{d\geq 1}\wO_d$).  
We denote by $\mS_d$ and $\wS_d$ respectively the set of orientations which are the image by duality of the sets $\mO_d$ and  $\wO_d$. 
Thus, $\mS_d$ is the set of vertex-rooted orientations which are accessible, have a root-vertex of indegree 0 
and degree $d$, and which are maximal for one of the faces incident to the root-vertex (maximality holds for each of these faces in this case); see Figure~\ref{fig:duality}(b).  The set $\wS_d$ is the subset of  $\mS_d$ made of the orientations such that all the faces incident to the root-vertex have \emph{counterclockwise degree} 1 (only one edge is oriented in counterclockwise direction around these faces). 

\begin{figure}[h] 
\begin{center} 
\includegraphics[width=\linewidth]{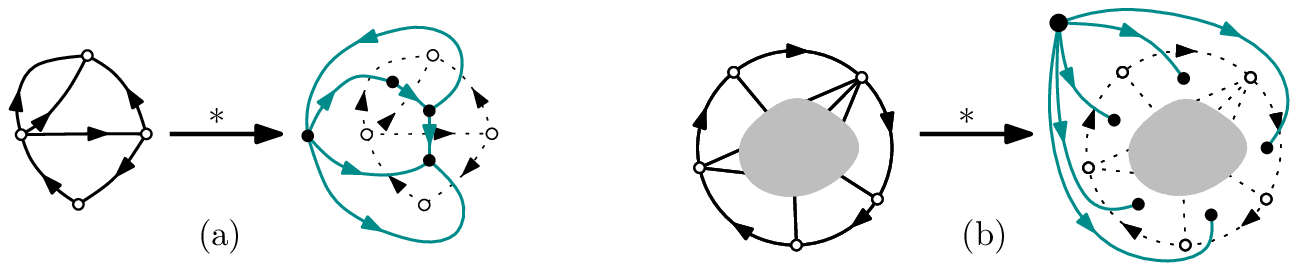} 
\end{center} 
\caption{The dual of an orientation.} 
\label{fig:duality} 
\end{figure} 
 
 
 
 

\titre{Partial closure and partial opening.}\\ 
We now present the bijection $\Psi$ from~\cite{OB:boisees} between corner-rooted maximal accessible orientations and mobiles of excess 1. The description of $\Psi$ below is in fact taken from~\cite[Section 7]{OB:covered-maps}.

Let $M$ be a mobile with  $p$ edges and $q$ buds (hence excess $\delta=p-q$). The corresponding \emph{fully blossoming} mobile $M'$ is obtained from $M$ by inserting a dangling half-edge called a \emph{stem}  in each of the $p$ corners of $M$ following an edge in counterclockwise direction around a black vertex. A fully blossoming mobile is represented in solid lines in Figure~\ref{fig:close_mobile_ccw}(b), where buds are (as usual) represented by outgoing arrows 
and stems are represented by ingoing arrows. 
A \emph{counterclockwise walk} around $M'$ (with the edges of the mobile $M'$ on the left of the walker) sees a succession of buds and stems. 
Associating an opening parenthesis to each bud and a closing parenthesis to each stem, 
one obtains a cyclic binary word with $q$ opening and $p$ closing parentheses. This 
yields in turn a partial matching of the buds with stems (a bud is matched with the next free stem in 
counterclockwise order around the mobile), leaving $|\delta|$ dangling half-edges unmatched (the unmatched 
dangling half-edges are stems if $\delta> 0$ and are buds if $\delta< 0$). 
 
The \emph{partial closure} $C$ of  the mobile $M$  is obtained by 
forming an oriented edge out of each matched pair. Clearly the oriented edges can be formed in a planar way, and this process is uniquely defined (recall that all our maps are on the sphere). The partial closure is represented in Figure~\ref{fig:close_mobile_ccw}(a)-(b). We consider the partial closure $C$ as a planar map with two types of edges (those of the mobile, which are non-oriented, and the new formed edges, which are oriented) and $|\delta|$ dangling half-edges which are all incident to the same face, which we take to be the root-face of $C$. Note that, if $\delta\geq 0$, there are $\delta$ white corners incident to the root-face of $C$, because initially the number of such corners is equal to the number $p$ of edges of the mobile, and then each matched pair forming an edge
(there are $q$ such pairs, one for each bud) decreases this number by $1$. These corners, which stay incident to the root-face throughout the partial closure, are called \emph{exposed white corners}.\\ 
 
\begin{figure} 
\begin{center} 
\includegraphics[width=\linewidth]{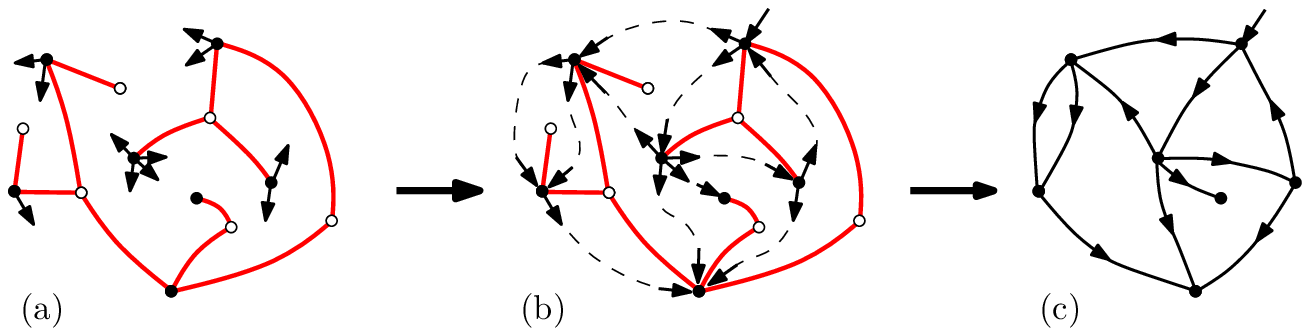} 
\end{center} 
\caption{The rooted closure of a mobile of excess $\delta=1$.} 
\label{fig:close_mobile_ccw} 
\end{figure}

Let $O$ be an oriented map. 
The \emph{partial opening} of $O$ is the map $C$ with vertices colored black or white and with two types of edges (oriented and non-oriented) obtained as follows. 
\begin{itemize} 
\item Color in black the vertices of $O$ and insert a white vertex in each face of $O$. 
\item Around each vertex $v$ of $O$ draw a non-oriented edge from any corner $c$ which follows an ingoing edge in clockwise order around $v$ to the white vertex in the face containing $c$. 
\end{itemize} 
If $O$ is corner-rooted, then the ingoing arrow indicating the root-corner is interpreted as an ingoing half-edge (a stem) and gives rise to an edge of $C$. For instance, the partial opening of the corner-rooted map in Figure~\ref{fig:close_mobile_ccw}(c) is the map in  Figure~\ref{fig:close_mobile_ccw}(b).\\

\titre{The rooted, positive and, negative opening/closure}\\ 
We now recall and extend the results given in~\cite{OB:boisees} about closures and openings. Observe that the partial closure $C$ of a mobile $M$ of excess $1$ has one dangling stem.  The \emph{rooted closure} of $M$, denoted $\Psi(M)$, is obtained from the partial closure $C$ by erasing every white vertex and every edge of the mobile; see Figure~\ref{fig:close_mobile_ccw}(b)-(c). The embedded graph $\Psi(M)$ is always connected (hence a map) and the dangling stem is considered as indicating the root-corner of $\Psi(M)$.  The \emph{rooted opening} of a corner-rooted orientation $O$ is obtained from its partial opening $C$ by erasing all the ingoing half-edges of $O$ (this leaves only the non-oriented edges of $C$ and some buds incident to black vertices). The following result was proved in~\cite{OB:boisees} (see also~\cite{OB:covered-maps}).

\begin{theo}\label{thm:base-case} 
The rooted closure $\Psi$ is a bijection between mobiles of excess $1$ and corner-rooted  maximal accessible orientations. The rooted opening is the inverse mapping. 
\end{theo}


We now present the mappings $\Psi_+$ and $\Psi_-$ defined respectively on mobiles of positive and negative excess, see Figure~\ref{fig:positive-closure}. 
\begin{Def} 
Let $M$ be a mobile of excess $\delta\neq 0$ and let $C$ be its partial closure. 
\begin{itemize} 
\item If $\delta>0$, then $C$ has $\delta$ stems (incident to the root-face). The positive closure of $M$, denoted $\Psi_+(M)$, is obtained from $C$ by first creating a (black) root-vertex $v$ of $\Psi_+(M)$ in the root-face of $C$ and connecting it to each stem (which becomes part of an edge of $O$ oriented away from $v$); second erasing the white vertices and edges of the mobile. 
\item If  $\delta<0$, then $C$ has $\delta$ buds (incident to the root-face).  The negative closure of $M$, denoted $\Psi_-(M)$, is obtained from $C$ by first creating a (black) root-vertex $v$ of $\Psi_-(M)$ in the root-face of $C$ and connecting it to each bud and then reorienting these edges (each bud becomes part of an edge of $\Psi_-(M)$ oriented away from $v$); second erasing the white vertices and edges of the mobile. 
\end{itemize} 
\end{Def} 
\fig{width=12cm}{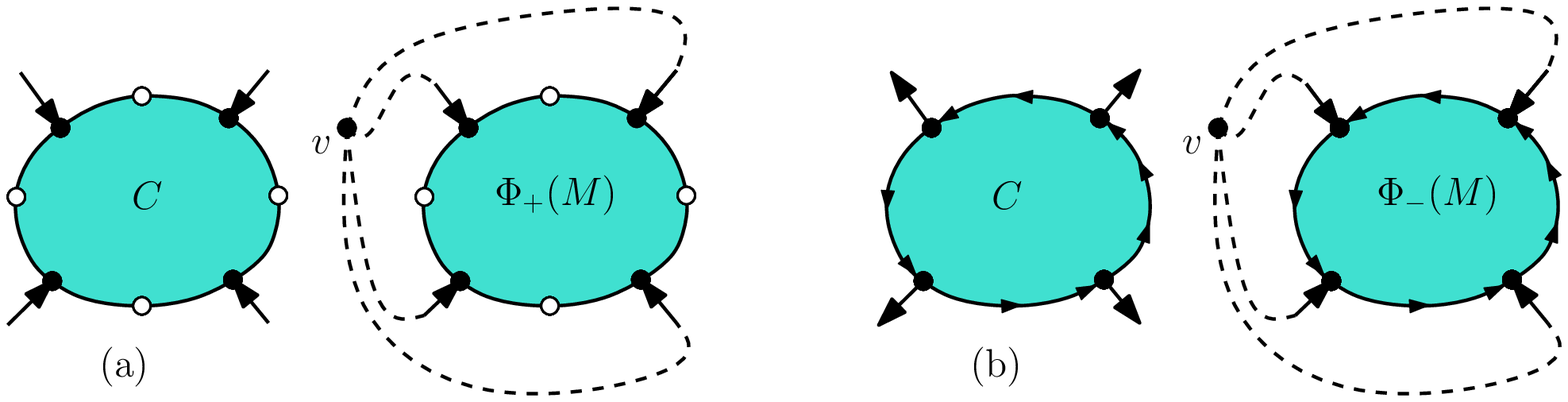}{(a) Positive closure $\Psi_+$. (b) Negative closure $\Psi_-$.} 
 

\begin{theo}\label{thm:delta_close} 
Let $d$ be a positive integer. 
\begin{itemize} 
\item The positive closure $\Psi_+$ is a bijection between the set of mobiles of excess $d$ and the set $\mS_{d}$ of orientations. Moreover, the mapping defined on $\mO_{d}$ by first applying duality (thereby obtaining an orientation in $\mS_d$) and then applying the inverse mapping $\Psi_+^{-1}$ is the mapping $\Phi_+$ defined in Definition~\ref{def:master-bijections}. Thus, $\Phi_+$ is a bijection between $\mO_{d}$ and (properly bicolored) mobiles of excess $d$. 
\item The negative-closure $\Psi_-$ is a bijection between the set of mobiles of excess $-d$ and the subset $\wS_{d}$ of orientations. Moreover, the mapping defined on $\wO_{d}$ by first applying duality  (thereby obtaining an orientation in $\wS_{d}$) and then applying the inverse mapping $\Psi_-^{-1}$ is the mapping $\Phi_-$ defined in Definition~\ref{def:master-bijections}. Thus, $\Phi_-$ is a bijection between $\wO_{d}$ and (properly bicolored) mobiles of excess $-d$. 
\end{itemize} 
\end{theo} 
 
Theorem~\ref{thm:delta_close} clearly implies Theorem~\ref{thm:master-bijections}. It only remains to prove Theorem~\ref{thm:delta_close}.  

\begin{proof} 
\ite $\ $ We first treat the case of the positive closure $\Psi_+$. We begin by showing that the positive closure of a mobile $M$ of excess $d$ is in $\mS_{d}$. Let $C$ be the partial closure of $M$ and let $O=\Psi_+(M)$ be its positive closure. As observed above, the mobile $M$ has $d>0$ exposed white corners. Let $M'$ be the  mobile obtained from $M$ by creating a new black vertex $b$, joining $b$ to an (arbitrary) exposed white corner, and adding $d$ buds to $b$;  see Figure~\ref{fig:delta_close}. The excess of $M'$ is 1, hence by Theorem~\ref{thm:base-case} the rooted closure $O'=\Psi(M')$ is maximal and accessible. Moreover, it is easily seen that $b$ is the root-vertex of $O'$ (because the stem incident to $b$ is not matched during the partial closure) and has indegree 0, see Figure~\ref{fig:delta_close}. 
Thus, the orientation $O$ is in $\mS_{d}$. 
 
We also make the following observation (useful for the negative-closure):\\ 
\textbf{Fact.} The positive closure $O=\Psi_+(M)$ is in $\wS_{d}$ if and only if each of the exposed white corners of $M$ is incident to a (white) leaf of $M$.\\ 
Indeed, a white vertex $w_f$ of $M$ has an exposed white corner if and only if it corresponds to a face $f$ of $O$ incident to the root-vertex~$b$. Moreover, the counterclockwise degree of $f$ is the degree of $w_f$.

\begin{figure} 
\begin{center} 
\includegraphics[width=10cm]{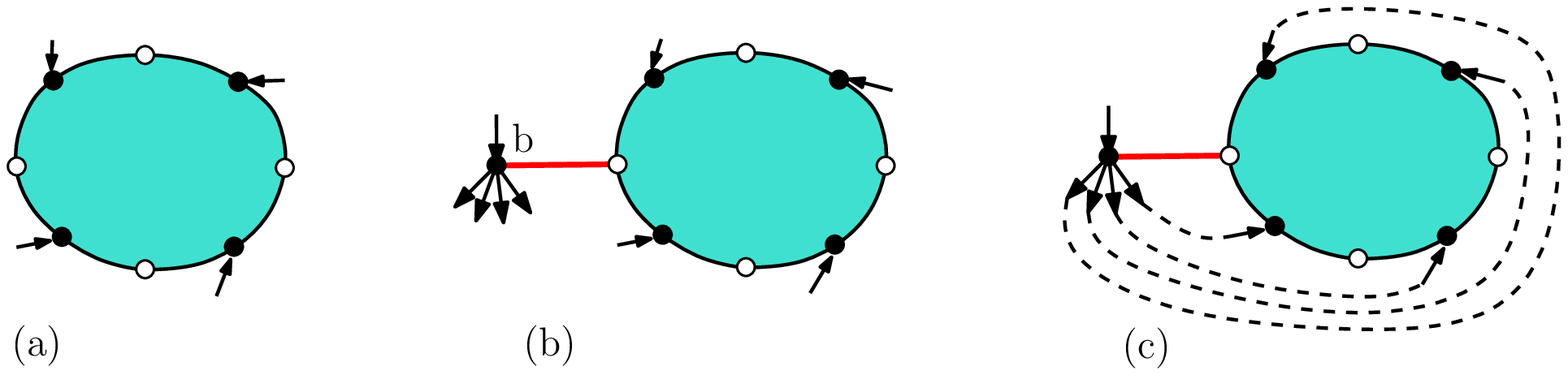} 
\end{center} 
\caption{Formulation of the positive closure $\Psi_+$ by reduction to the rooted closure $\Psi$. 
Figure (a) shows generically the partial closure of a mobile of excess $d=4$. In (b) one creates 
a black vertex $b$ with $d$ buds, and connects it to an exposed white corner. In (c) one performs the remaining matchings of buds with stems to complete the rooted closure.} 
\label{fig:delta_close} 
\end{figure} 
 
\begin{figure} 
\begin{center} 
\includegraphics[width=10cm]{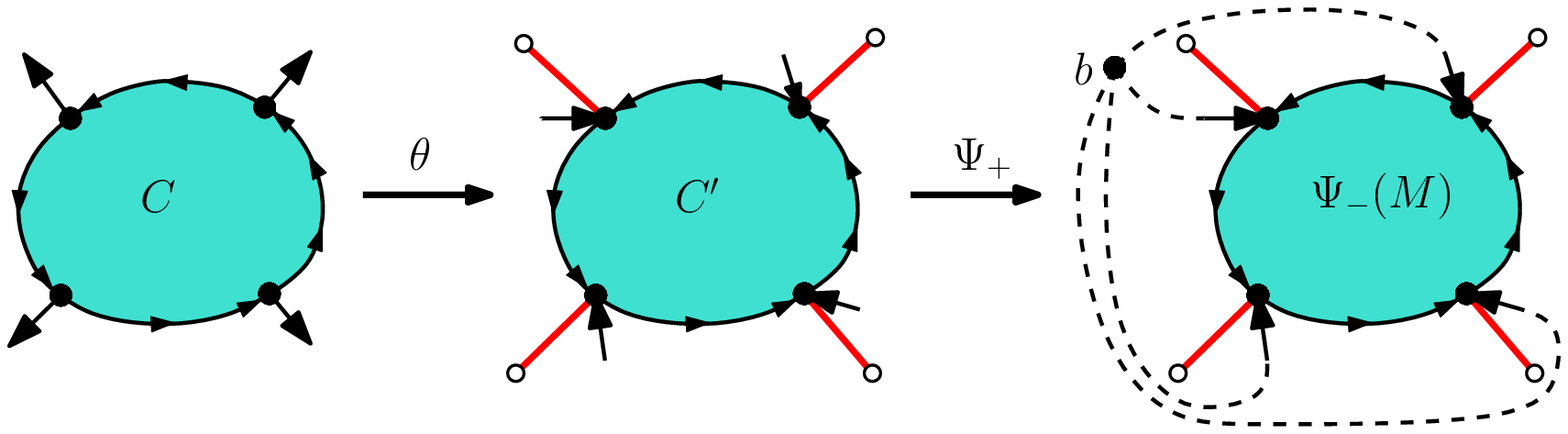} 
\end{center} 
\caption{Formulation of the negative closure $\Psi_-$, by reduction to the positive closure $\Psi_+$.} 
\label{fig:delta_close2} 
\end{figure} 
 
We now prove that the positive closure is a bijection by defining the inverse mapping. Let $O$ be a vertex-rooted orientation in $\mS_d$. We define the \emph{positive opening} of $O$, as the embedded graph with buds $M$ obtained by applying the partial opening of $O$, and then erasing every ingoing half-edge of $O$ as well as the root-vertex $b$ (and the incident outgoing half-edges).   In order to prove that $M$ is a  mobile, we consider a maximal accessible orientation $O'$ obtained from $O$ by choosing a root-corner for $O$ among the corners incident to the root-vertex $b$. By Theorem~\ref{thm:base-case}, the rooted opening of $O'$ gives a mobile $M'$. It is clear from the definitions that  $M$ is obtained from $M'$ by erasing the root-vertex $b$. Moreover, $b$ is a leaf of $M'$ (since $b$ is incident to no ingoing half-edge except the stem indicating the root-corner of $O$), hence $M$ is a mobile. Moreover, since the rooted closure and rooted opening are inverse bijections, it is clear that the positive closure and the positive opening are inverse bijections.

Lastly, it is clear from the definitions that taking an orientation $O$ in $\mO_d$, and applying the positive opening $\Psi_+^{-1}$ to the dual orientation $O^*$ gives the mobile $\Phi_+(O)$ as defined in Definition~\ref{def:master-bijections}. 
\smallskip

\ite $\ $ We now treat the case of the negative closure $\Psi_-$. Let $M$ be a mobile of excess $-d$. 
We denote by $\theta(M)$ the mobile of excess $d$ obtained from $M$ by transforming each of its $d$ unmatched  buds into an edge connected to a new white leaf. It is clear (Figure ~\ref{fig:delta_close2}) that the positive closure of $\theta(M)$ is equal to the negative closure of $M$, hence $\Psi_-=\Psi_+\circ\theta$. Moreover, $\theta$ is clearly a bijection between mobiles of excess $-d$ and mobiles of excess $d$ such that every exposed white corner belongs to a leaf (the inverse mapping $\theta^{-1}$ replaces each edge incident to an exposed leaf by a bud). By the fact underlined above, this shows that   $\Psi_-=\Psi_+\circ\theta$ is a bijection between  mobiles of excess $-d$ and the set $\wS_d$. 
 
Lastly, denoting by $*$ the duality mapping, one gets $\Psi_-^{-1}\circ *=\theta^{-1}\circ\Psi_+^{-1}\circ *=\theta^{-1}\circ\Phi_+$. In other words, taking an orientation $O$ in $\wO_d$, and applying the negative opening $\Psi_-^{-1}$ to $O^*$, is the same as applying the mapping $\Phi_+$ and then replacing each edge of the mobile incident to an outer vertex by a bud. By definition, this is the same as applying the mapping $\Phi_-$ to $O$. This completes the proof of Theorem~\ref{thm:delta_close} and of Theorem~\ref{thm:master-bijections}. 
\end{proof}


\section{Opening/closure for mobiles of excess $0$} 
We include here a mapping denoted $\Phi_0$ for mobiles of excess $0$, which can be seen as the mapping $\Phi_-$ when the outer face is degenerated to a single vertex (seen as a face of degree $0$). This mapping makes it possible to recover a bijection by Bouttier, Di Francesco, and Guitter~\cite{BDFG:mobiles}, as we will explain in~\cite{BeFu_Girth}.   

We call \emph{source-orientation} (resp. \emph{source-biorientation}) a vertex-rooted orientation (resp. biorientation) accessible from the root-vertex and such that all half-edges incident to the root-vertex are outgoing. 
Let $\wB_0$ be the set of weighted source-biorientations that are minimal with respect 
to a face incident to the root-vertex (in which case minimality holds with respect 
to all faces incident to the root-vertex).  
 
\begin{Def}\label{def:Phi0} 
Let $B\in\wB_0$, with root-vertex $v_0$. We view $B$ as embedded on the sphere, with the vertices of $B$ 
colored white, and we place a black vertex $b_f$ in each face $f$. The embedded 
graph $\Phi_0(B)$ is obtained by performing the local transformation of Figures~\ref{fig:one-way-operation} 
and~\ref{fig:i-way-operation}  
on each edge of $B$ (with weight transfer rules) and then deleting $v_0$.  
\end{Def} 
 
The mapping $\Phi_0$ is illustrated in Figure~\ref{fig:Phi0}. 
 
\begin{figure} 
\begin{center} 
\includegraphics[width=12.6cm]{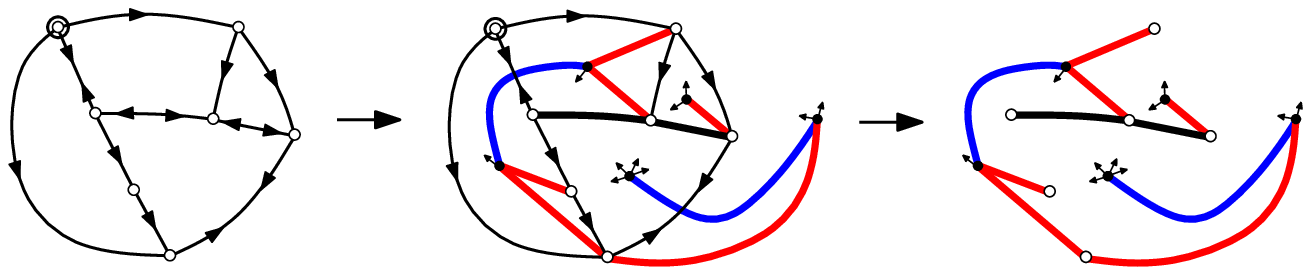} 
\end{center} 
\caption{Master bijection $\Phi_0$ applied to a source-biorientation in $\wB_0$ (weights are not indicated).} 
\label{fig:Phi0} 
\end{figure} 
 
\begin{theo}\label{theo:Phi0} 
The mapping $\Phi_0$ is a bijection between $\wB_0$ and weighted mobiles of excess zero. 
\end{theo} 
 
By the reduction argument shown in Figure~\ref{fig:bijection_biorientation}  
(straightforwardly adapted to source-biorientations),  
 proving Theorem~\ref{theo:Phi0} comes down to proving the bijective result for ordinary unweighted orientations, 
that is, we have to show that the set $\wOO$ of source-orientations that are minimal 
with respect to a face incident to the pointed vertex (in which case minimality 
holds for any face incident to the pointed vertex) is in  
bijection via $\Phi_0$ with the set of properly bicolored mobiles of excess $0$.  
Note that $\wOO$ is in bijection (via duality) with the set $\wSO$  
of counterclockwise-maximal (i.e., maximal and such that the contour 
of the outer face is a counterclockwise circuit) 
 accessible orientations. Similarly as in the previous section, we define 
two mappings, respectively called zero-closure and zero-opening, which establish a  
bijection between (properly bicolored) mobiles of excess $0$ and orientations in $\wSO$.  
Given a mobile $M$ of excess $0$, let $C$ be the partial closure of $M$. Then the  
\emph{zero-closure} of $M$, denoted $\Psi_0(M)$, is obtained from $C$ by erasing the  
 white vertices and the edges of~$M$.  
 
\begin{theo} 
The zero-closure $\Psi_0$ is a bijection between (properly bicolored) mobiles of excess $0$ 
and the set $\wSO$ of counterclockwise-maximal accessible orientations. Moreover 
the mapping defined on $\wOO$ by first applying duality (thereby obtaining an orientation in $\wSO$)   
 and then applying $\Psi_0^{-1}$ is the mapping $\Phi_0$ defined in Definition~\ref{def:Phi0}.  
\end{theo} 
\begin{proof} 
Let $M$ be a mobile of excess $0$, let $C$ be the partial closure of $M$, and $O$ the zero-closure of $M$.   As observed in Section \ref{sec:fromOBtoPhi}, $C$ has no exposed white corner (white corner incident to the root-face) hence it has some \emph{exposed black corners} (black corners incident to the root-face). 
Choose an exposed black corner, and let $M'$ be the mobile of excess $1$ obtained 
from $M$ by attaching a new black-white edge (connected to a new white leaf) at the chosen exposed 
black corner. Let $O'$ be the rooted closure of $M'$. Clearly the face-rooted 
orientation $O$ is induced by the corner-rooted orientation $O'$ (and each choice 
of exposed black corner yields each of the corner-rooted orientations inducing $O$).   
Moreover, $O'$ is maximal accessible by Theorem \ref{thm:base-case}.
Lastly,   since the white vertex carrying the dangling stem ---called the \emph{root white vertex} 
of the mobile--- is incident to a unique edge, 
the root-face of $O'$ (hence also of $O$) is counterclockwise.
Thus, $O$ is counterclockwise-maximal accessible, that is,  $O$ is in $\wSO$. 
 
We now define the inverse mapping, called zero-opening. Given $O\in\wSO$, the zero-opening of $O$  
is the embedded graph with buds obtained by first computing the partial opening of $O$,  
 then erasing all ingoing half-edges of $O$, and then erasing the white vertex placed 
in the root-face of $O$.        
It is easily checked that, for any corner-rooted orientation $O'$ inducing $O$, 
the rooted opening $M'$ of $O'$ is a (properly bicolored) mobile of excess $1$ 
where the root white vertex is incident to a unique edge. In addition the deletion of this edge 
 gives $M$, so $M$ is a (properly bicolored) mobile of excess $0$. To conclude, since the rooted 
closure and rooted opening are mutually inverse, then the zero closure and zero opening are also mutually inverse. 
Hence $\Psi_0$ is a bijection.   
Finally, the fact that $\Phi_0$ coincides with duality followed by $\Psi_0^{-1}$ is justified in the same 
way as in Theorem~\ref{thm:delta_close}.  
\end{proof}

\bigskip 
 
\section{Additional remarks} 
We have presented a general bijective strategy for planar maps that relies 
on certain orientations with no counterclockwise circuit. 
We have applied the approach to an infinite collection 
$(\mC_{d,p})_{d\geq 3,p\geq d}$ of families planar maps, where $\mC_{d,p}$ denotes 
the set of $d$-angulations of girth $d$ with a boundary of size $p$. 
For this purpose we have introduced so-called \emph{\ddm-orientations} for $d$-angulations of girth $d$. 
In future work we shall further explore and exploit the properties of these orientations: 
 
\titre{Schnyder decompositions.} In \cite{BeFu_Schnyder} we show that the \ddm-orientations of a $d$-angulation of girth $d$ are in bijection with certain coverings of the $d$-angulation by a set of $d$ forests crossing each other in a specific manner. These forest coverings extend to arbitrary $d\geq 3$ the so-called \emph{Schnyder woods} corresponding to the case $d=3$~\cite{Schnyder:wood1}.\\
\titre{Extension of the bijections to planar maps of fixed girth.} 
For each integer  $d\geq 3$, we have presented in Section~\ref{sec:bij_dang} a bijection for the class of $d$-angulations of girth $d$, which consists in a certain specialization of the master bijection $\Phi_-$. 
In the article~\cite{BeFu_Girth} we show that this bijection can be extended to the class $\mG_d$ of all  maps of girth $d$. 
The strategy in the article~\cite{BeFu_Girth} parallels the one initiated here: we characterize the maps in $\mG_d$ by certain (weighted) orientations and then obtain a bijection by specializing the master bijection $\Phi_-$. The bijection obtained for $\mG_d$ associates to any map of girth $d$ a mobile such that an inner face of degree~$i$ in the map corresponds to a black vertex of degree~$i$  in the mobile, so it gives a way of counting maps with control both on the girth and the face-degrees.\\

\bibliographystyle{alpha} 
\bibliography{biblio-dangulations} 
\label{sec:biblio} 
 
\end{document}